\documentclass[oneside,english]{amsart}
\usepackage[T1]{fontenc}
\usepackage[latin9]{inputenc}
\usepackage{mathrsfs}
\usepackage{amstext}
\usepackage{amsthm}
\usepackage{amssymb}

\makeatletter
\numberwithin{equation}{section}
\numberwithin{figure}{section}
\theoremstyle{plain}
\newtheorem{thm}{\protect\theoremname}
\theoremstyle{definition}
\newtheorem{defn}[thm]{\protect\definitionname}
\theoremstyle{plain}
\newtheorem{lem}[thm]{\protect\lemmaname}
\theoremstyle{plain}
\newtheorem{prop}[thm]{\protect\propositionname}
\theoremstyle{remark}
\newtheorem{rem}[thm]{\protect\remarkname}
\theoremstyle{plain}
\newtheorem{cor}[thm]{\protect\corollaryname}
\theoremstyle{definition}
\newtheorem{example}[thm]{\protect\examplename}

\makeatother

\usepackage{babel}
\providecommand{\corollaryname}{Corollary}
\providecommand{\definitionname}{Definition}
\providecommand{\examplename}{Example}
\providecommand{\lemmaname}{Lemma}
\providecommand{\propositionname}{Proposition}
\providecommand{\remarkname}{Remark}
\providecommand{\theoremname}{Theorem}

\begin{document}
\global\long\def\G{\mathcal{G}}%
 
\global\long\def\F{\mathcal{F}}%
 
\global\long\def\H{\mathcal{H}}%
\global\long\def\Z{\mathcal{Z}}%
 
\global\long\def\L{\mathcal{L}}%
\global\long\def\U{\mathcal{U}}%
\global\long\def\W{\mathcal{W}}%
 
\global\long\def\E{\mathcal{E}}%
\global\long\def\B{\mathcal{B}}%
 
\global\long\def\A{\mathcal{A}}%
\global\long\def\D{\mathcal{D}}%
\global\long\def\O{\mathcal{O}}%
 
\global\long\def\N{\mathcal{N}}%
 
\global\long\def\X{\mathcal{X}}%
 
\global\long\def\lm{\lim\nolimits}%
 
\global\long\def\then{\Longrightarrow}%
\global\long\def\id{\operatorname{id}}%

\global\long\def\V{\mathcal{V}}%
\global\long\def\C{\mathcal{C}}%
\global\long\def\adh{\operatorname{adh}}%
\global\long\def\Seq{\operatorname{Seq}}%
\global\long\def\intr{\operatorname{int}}%
\global\long\def\cl{\operatorname{cl}}%
\global\long\def\inh{\operatorname{inh}}%
\global\long\def\diam{\operatorname{diam}\ }%
\global\long\def\card{\operatorname{card}}%
\global\long\def\T{\operatorname{T}}%
\global\long\def\S{\operatorname{S}}%
\global\long\def\I{\operatorname{I}}%
\global\long\def\K{\operatorname{K}}%
\global\long\def\rdc{\operatorname{rdc}}%

\global\long\def\fix{\operatorname{fix}}%
\global\long\def\Epi{\operatorname{Epi}}%
\global\long\def\UU{\mathfrak{U}}%

\global\long\def\FF{\mathfrak{F}}%
\global\long\def\GG{\mathfrak{G}}%
\global\long\def\cont{\mathscr{C}}%
\global\long\def\conv{\mathsf{Conv}}%
\global\long\def\prtop{\mathsf{PrTop}}%
\global\long\def\top{\mathsf{Top}}%
\global\long\def\pconv{\mathsf{pConv}}%
\global\long\def\pstop{\mathsf{PsTop}}%
\global\long\def\Cap{\mathsf{Cap}}%
\global\long\def\prap{\mathsf{PrAp}}%
\global\long\def\ap{\mathsf{Ap}}%
\global\long\def\psap{\mathsf{PsAp}}%
\global\long\def\Cconv{\mathsf{C^{Conv}}}%

\title{On some convergence approach structures on hyperspaces}
\author{M.Ate\c{s}$^{*}$, F. Mynard$^{\dagger}$, S.Sa\u{g}\i ro\u{g}lu$^{\ddagger}$}
\begin{abstract}
In the context of the category $\Cap$ of convergence approach spaces
and contractions, we introduce and study approach analogs of the upper
and lower Kuratowski convergences, upper-Fell and Fell topologies
on the set of closed subsets of the coreflection on the category $\mathsf{Conv}$
of convergence spaces of a convergence approach space. In particular,
over a pre-approach space, the $\conv$-coreflection of the lower
Kuratowski convergence approach structure is the lower Kuratowski
convergence associated with the $\conv$-coreflection of the base
space, while the $\conv$-reflection is the lower Kuratowski convergence
associated with the $\conv$-reflection. The $\conv$-coreflection
of the upper Kuratowski convergence approach is is the upper Kuratowski
convergence associated with the $\conv$-reflection of the base space,
while the $\conv$-reflection is the upper Kuratowski convergence
associated with the $\conv$-coreflection of the base space. We show
that, over an approach space, the lower Kuratowski convergence approach
structure is in fact an approach structure that coincides with the
$\vee$-Vietoris approach structure introduced by Lowen and his collaborators,
though it may be strictly finer over a general convergence approach
space. We show that the upper Fell convergence approach structure
is a non-Archimedean approach structure coarser than the upper Kuratowski
convergence approach, but finer than the upper Fell approach structure
introduced by the first and third author. We also obtain a $\Cap$
abstraction of the classical result that if the upper Kuratowski convergence
over a topological space is pretopological, then it is also topological.
\end{abstract}

\email{$^{*}$mbiten@ankara.edu.tr, $^{\dagger}$fmynard@njcu.edu, $^{\ddagger}$ssagir@science.ankara.edu.tr}
\address{$^{*,\ddagger}$Department of Mathematics, Faculty of Science, Ankara
University, 06100, Ankara, TÜRK\.{I}YE}
\address{$^{\dagger}$College of Arts and Sciences, New Jersey City University,
2039 Kennedy Blvd, Jersey City, NJ 07305, USA}
\maketitle

\section{PREL\.{I}M\.{I}NAR\.{I}ES}

If $X$ is a set, we denote its powerset by $\mathbb{P}X$ and its
set of finite subsets by $[X]^{<\infty}$. The set of (set-theoretic)
filters on $X$ is denoted $\mathbb{F}X$ and the set of ultrafilters
on $X$ is denoted $\mathbb{U}X$. If $\A\subset\mathbb{P}X$ is a
family of subsets of $X$, we will consider the following associated
families: its \emph{isotone hull}
\[
\A^{\uparrow}=\{B\subset X:\exists A\in\A,A\subset B\},
\]
its \emph{grill
\[
\A^{\#}=\left\{ B\subset X:\forall A\in\A,B\cap A\neq\emptyset\right\} ,
\]
}its \emph{completion by finite intersections
\[
\A^{\cap}=\{\bigcap\B:\B\in[\A]^{<\infty}\}.
\]
}

When $\A=\{A\}$, we abridge $A^{\uparrow}$ for the principal filter
$\{A\}^{\uparrow}$ of $A$. In particular, if $x\in X$ then $\{x\}^{\uparrow}$
denotes the principal ultrafilter generated by $x$. 

If $\A,\B\subset\mathbb{P}X$, we say that $\A$ and $\B$ \emph{mesh},
in symbols $\A\#\B$, if $\A\subset\B^{\#}$, equivalently, $\B\subset\A^{\#}$.
We say that $\A$ \emph{is finer than $\B$ }or that $\B$ \emph{is
coarser than $\A$, }in symbols\emph{ $\A\geq\B$, }if for every $B\in\B$
there is $A\in\A$ with $A\subset B$, that is, if $\B\subset\A^{\uparrow}$.
For every filter $\F$ on $X$, $\beta(\F)=\{\U\in\mathbb{U}X:\U\geq\F\}$
denotes the set of finer ultrafilters. When $\F=A^{\uparrow}$ is
principal, we abridge to $\beta A=\{\U\in\mathbb{U}X:A\in\U\}$.

A \emph{convergence }$\xi$ on a set $X$ is a relation between points
of $X$ and filters on $X$, denoted $x\in\lm_{\xi}\F$ if $(x,\F)\in\xi$
and interpreted as $x$ is a limit point for $\F$ in $\xi$, satisfying
\begin{equation}
\tag{centered}x\in\lm_{\xi}\{x\}^{\uparrow}\label{eq:centeredconv}
\end{equation}
for every $x\in X$ and 
\begin{equation}
\tag{monotone}\F\leq\G\then\lm_{\xi}\F\subset\lm_{\xi}\G,\label{eq:convmontone}
\end{equation}
for all $\F,\G\in\mathbb{F}X$. The pair $(X,\xi)$ is then called
a \emph{convergence space. }If only (\ref{eq:convmontone}) is satisfied
but not necessarily (\ref{eq:centeredconv}), $(X,\xi)$ is a called
a \emph{preconvergence space}.

A map $f:X\to Y$ between two (pre)convergence spaces $(X,\xi)$ and
$(Y,\tau)$ is \emph{continuous}, in symbols $f\in\cont(\xi,\tau)$,
if $f(x)\in\lm_{\tau}f[\F]$ whenever $x\in\lm_{\xi}\F$, where $f[\F]=\{f(F):F\in\F\}^{\uparrow}$
is the image filter. Let $\pconv$ and $\conv$ denote the category
of preconvergence spaces and continuous maps and convergence spaces
and continuous maps, respectively. If $\xi,\sigma$ are two (pre)convergences
on $X$, we say that $\xi$ is \emph{finer than $\sigma$ }or that
$\sigma$ is \emph{coarser than $\xi$ }if the identity map $\id_{X}\in\cont(\xi,\sigma)$,
that is, if $\lim_{\xi}\F\subset\lim_{\sigma}\F$ for every filter
$\F$ on $X$.

If a convergence additionally satisfies
\begin{equation}
\tag{finite depth}\lm_{\xi}(\F\cap\G)=\lm_{\xi}\F\cap\lm_{\xi}\G\label{eq:fintedepth}
\end{equation}
 for every $\F,\G\in\mathbb{F}X$, we say that $\xi$ \emph{has finite
depth} (\footnote{Many authors include (\ref{eq:fintedepth}) in the axioms of a convergence
space. Here we follow \cite{DM.book}.}). A convergence satisfying the stronger condition that 
\begin{equation}
\tag{pretopology}\lm_{\xi}(\bigcap_{\D\in\mathbb{D}}\D)=\bigcap_{\D\in\mathbb{D}}\lm_{\xi}\D\label{eq:pretopdeep}
\end{equation}
 for every $\mathbb{D}\subset\mathbb{F}X$ is called a \emph{pretopology}
(and $(X,\xi)$ is a \emph{pretopological space}). 

A subset $A$ of a convergence space $(X,\xi)$ is \emph{$\xi$-open
}if 
\[
\lm_{\xi}\F\cap A\neq\emptyset\then A\in\F,
\]
and $\xi$-\emph{closed} if it is closed for limits, that is,
\[
A\in\F\then\lm_{\xi}\F\subset A,
\]
equivalently,
\[
A\in\F^{\#}\then\lm_{\xi}\F\subset A.
\]

Let $\O_{\xi}$ denote the set of open subsets of $(X,\xi)$ and let
$\O_{\xi}(x)=\{U\in\O_{\xi}:x\in U\}$. Similarly, let $\mathcal{C}_{\xi}$
denote the set of closed subsets of $(X,\xi)$. It turns out that
$\O_{\xi}$ is a topology on $X$\emph{. }Moreover,\emph{ }a topology
$\tau$ on a set $X$ determines a convergence $\xi_{\tau}$ on $X$
by
\[
x\in\lm_{\xi_{\tau}}\F\iff\F\geq\N_{\tau}(x),
\]
where $\N_{\tau}(x)$ denotes the neighborhood filter of $x$ for
$\tau$. In turn, $\xi_{\tau}$ completely determines $\tau$ because
$\tau=\O_{\xi_{\tau}}$, so that we do not distinguish between $\tau$
and $\xi_{\tau}$ and identify topologies with special convergences.
Moreover, a convergence $\xi$ on $X$ determines the topology $\O_{\xi}$
on $X$ which turns out to be the finest among the topologies on $X$
that are coarser than $\xi$. We call it the \emph{topological modification
of $\xi$ }and denote it $\T\xi$. Hence the category $\top$ of topological
spaces and continuous maps is a concretely reflective subcategory
of $\conv$. Let $\cl_{\xi}$ denote the closure operator in $\T\xi$.

A convergence $\xi$ determines an \emph{adherence} operator on families
$\F\subset\mathbb{P}X$, given by
\begin{equation}
\adh_{\xi}\F=\bigcup_{\mathbb{F}X\ni\G\#\F}\lm_{\xi}\G.\label{eq:adherencegen}
\end{equation}
If $\F\in\mathbb{F}X$, then 
\begin{equation}
\adh_{\xi}\F=\bigcup_{\mathbb{F}X\ni\G\#\F}\lm_{\xi}\G=\bigcup_{\G\geq\F}\lm_{\xi}\G=\bigcup_{\U\in\beta(\F)}\lm_{\xi}\U.\label{eq:adherencefilter}
\end{equation}

A convergence $\xi$ on $X$ is a \emph{pseudotopology }if 
\begin{equation}
\tag{pseudotopology}\lm_{\xi}\F=\bigcap_{\U\in\beta(\F)}\lm_{\xi}\U\label{eq:pseudotop}
\end{equation}
 for every filter $\F\in\mathbb{F}X$. The full subcategory of $\conv$
formed by pseudotopological spaces (and continuous maps), denoted
$\pstop$, is concretely reflective and the reflector $\S$ can be
described on objects by
\begin{equation}
\lm_{\S\xi}\F=\bigcap_{\U\in\beta(\F)}\lm_{\xi}\U=\bigcap_{\G\#\F}\adh_{\xi}\G,\label{eq:Sreflector}
\end{equation}
so that adherences determine the pseudotopological reflection.

When restricted to principal filters, (\ref{eq:adherencegen}) defines
a \emph{principal adherence operator }$\adh_{\xi}:\mathbb{P}X\to\mathbb{P}X$
given by
\[
\adh_{\xi}A=\adh_{\xi}\{A\}^{\uparrow}=\bigcup_{A\in\G^{\#}}\lm_{\xi}\G=\bigcup_{A\in\G}\lm_{\xi}\G=\bigcup_{\U\in\beta A}\lm_{\xi}\U.
\]

The full subcategory of $\conv$ formed by pretopological spaces (and
continuous maps), denoted $\prtop$, is concretely reflective and
the reflector $\S_{0}$ can be described on objects by
\begin{equation}
\lm_{\S_{0}\xi}\F=\bigcap_{A\in\F^{\#}}\adh_{\xi}A,\label{eq:S0adh}
\end{equation}
so that the principal adherence determines the pretopological reflection.
In general,
\[
\adh_{\xi}A\subset\cl_{\xi}A
\]
but not conversely. In contrast to $\cl_{\xi}$, the principal adherence
is in general non-idempotent because $\adh_{\xi}A$ need not be closed.
A pretopology is a topology if and only if its principal adherence
operator is idempotent, in which case $\adh_{\xi}=\cl_{\xi}$. Moreover,
the reflector $\T$ from $\conv$ onto $\top$ is given by
\begin{equation}
\lm_{\T\xi}\F=\bigcap_{A\in\F^{\#}}\cl_{\xi}A.\label{eq:Tclosure}
\end{equation}

We refer the reader to \cite{DM.book} for a systematic study of convergence
spaces.

\subsection{Hyperspace convergences and topologies\protect\label{subsec:Hyperspace-convergences-and}}

We gather here the basic necessary facts on hyperspace convergences
from \cite[Section VIII.7]{DM.book}. In a convergence space $(X,\xi)$,
the \emph{reduction }of a family $\FF\subset\mathbb{P}(\C_{\xi})$
is defined as
\[
\rdc\FF:=\left\{ \bigcup_{C\in\F}C:\F\in\FF\right\} .
\]
Note that if $\FF\in\mathbb{F}(\C_{\xi})$ and $\FF\neq\{\emptyset\}^{\uparrow}$,
then $(\rdc\FF)^{\uparrow}\in\mathbb{F}X$. 
\begin{defn}
Let $\FF\in\mathbb{F}(\C_{\xi})$ and \emph{$A\in\C_{\xi}$. }We say
that $\FF$\emph{ upper-Kuratowski converges to} $A$ if $\adh_{\xi}\rdc\FF\subset A$.
In symbols,
\begin{equation}
A\in\lm_{uK}\FF\iff\adh_{\xi}\rdc\FF\subset A.\label{eq:uKConv}
\end{equation}
On the other hand, $\FF$ \emph{lower-Kuratowski converges to $A$
}if $A\subset\adh_{\xi}\rdc(\FF^{\#}).$ In symbols,
\begin{equation}
A\in\lm_{lK}\FF\iff A\subset\adh_{\xi}\rdc(\FF^{\#}).\label{eq:lKconv}
\end{equation}
\end{defn}

Finally, $\FF$ \emph{Kuratowski converges to $A$} if $A\in\lm_{uK}\FF\cap\lm_{lK}\FF$,
in symbols,
\begin{equation}
\lm_{K}\FF=\lm_{uK}\FF\cap\lm_{lK}\FF.\label{eq:Kconv}
\end{equation}

The upper-Kuratowski convergence is always pseudotopological \cite[(VIII.7.4) and Theorem VIII.6.2]{DM.book}.
If $(X,\xi)$ is a topological space, then the lower-Kuratowski convergence
is topological \cite[Prop. VIII.7.5]{DM.book} and the Kuratowski
convergence is a Hausdorff pseudotopology \cite[Prop. VIII.7.7]{DM.book}.
\global\long\def\e{\operatorname{e}}%

If $F\subset X$, we denote the \emph{erected set} by $\e F=\{C\in\C_{X}:C\subset F\}$
and if $\F\in\mathbb{F}X$ then $e^{\natural}\F=\{eF:e\in\F\}$ is
a filter-base on $\C_{X}$ generating a filter called \emph{erected
filter of }$\F$. Note that 
\[
\F\in\mathbb{F}X\then\rdc(e^{\natural}\F)\geq\F
\]
though $\rdc(e^{\natural}\F)$ may be degenerate, and 
\[
\FF\in\mathbb{F}\C_{X}\then\FF\geq e^{\natural}(\rdc\FF).
\]

A subset $\A$ of $\C_{X}$ is \emph{saturated} if every closed subset
of $\rdc\A=\bigcup\A$ belongs to $\A$. A filter is \emph{saturated
}if it has a filter-base composed of saturated sets. Since every erected
set is saturated and $\rdc(e^{\natural}(\rdc\FF))=\rdc\FF$, we conclude
that any convergence structure on $\C_{X}$ solely defined in terms
of $\rdc\FF$ is based in saturated filters.

Recall that the \emph{cocompact topology} or \emph{upper-Fell topology}
on the set $\C_{X}$ of closed subsets of a topological space is generated
by the subbase given by sets of the form
\[
K^{+}=\{A\in\C_{X}:A\cap K=\emptyset\},
\]
where $K$ runs over compact subsets of $X$. The \emph{lower Vietoris
}topology is generated by the subbase given by sets of the form
\[
U^{-}=\C_{X}\cap U^{\#}=\{A\in\C_{X}:A\cap U\neq\emptyset\},
\]
where $U\in\O_{X}$ and coincides with the lower Kuratowski topology.
The \emph{Fell topology }is the supremum of these two topologies.

Note also:
\begin{lem}
\label{lem:grillrdc} If $X$ is a topological space, $\FF\in\mathbb{F}(\C_{X})$,
and $B\in\C_{X}$ then

\begin{equation}
B\in(\rdc\FF)^{\#}\iff B^{-}=\{C\in\C_{X}:C\cap B\neq\emptyset\}\in\FF^{\#};\label{eq:BingrillofrdcF}
\end{equation}
Moreover, if $\FF$ is a saturated filter and $\H$ is a saturated
subset of $\C_{X}$, then
\[
\H\in\FF^{\#}\iff\rdc\H\in(\rdc\FF)^{\#}.
\]

On the other hand, 
\begin{equation}
B\in(\rdc\FF^{\#})^{\#}\iff B^{-}\in\FF.\label{eq:BingrillofrdcFgrill}
\end{equation}
\end{lem}

\begin{proof}
$B\notin(\rdc\FF)^{\#}$ if and only if there is $\F\in\FF$ with
$\rdc\F\cap B=\emptyset$, equivalently, $\F\subset\{C\in\C_{X}:C\cap B=\emptyset\}$,
if and only if $\{C\in\C_{X}:C\cap B=\emptyset\}\in\FF$, equivalently,
$\{C\in\C_{X}:C\cap B\neq\emptyset\}\notin\FF^{\#}$. 

It is clear that $\rdc\H\in(\rdc\FF)^{\#}$ whenever $\H\in\FF^{\#}$.
If now $\H\notin\FF^{\#}$, $\H$ and $\FF$ are saturated, there
is a saturated $\F\in\FF$ with $\F\cap\H=\emptyset$ so that $\rdc\F\cap\rdc\H=\emptyset$
by saturation. 

Finally, $B\in(\rdc\FF^{\#})^{\#}$ if and only if $B\cap\rdc\H\neq\emptyset$
whenever $\H\in\FF^{\#}$, if and only if 
\[
\forall\H\in\FF^{\#},\exists C\in\H:B\cap C\neq\emptyset,
\]
if and only if $B^{-}\in\FF^{\#\#}=\FF$.
\end{proof}
Note that we have the following variant of \cite[Prop. VIII.7.5]{DM.book}:
\begin{prop}
\label{prop:pretoplK} If $(X,\xi)$ is a pretopological space then
$(\C_{\xi},\lim_{lK})$ is also a pretopological space and 
\[
\V_{lK}(A)=\bigvee_{x\in A}\{V^{-}:V\in\V_{\xi}(x)\}^{\uparrow},
\]
for every $A\in\C_{\xi}$.
\end{prop}

\begin{proof}
Recall that $A\in\lm_{lK}\FF$ if $A\subset\adh_{\xi}\rdc(\FF^{\#})$,
that is, if
\[
\forall x\in A,\V_{\xi}(x)\#\rdc(\FF^{\#}),
\]
because $\xi$ is a pretopology. In view of (\ref{eq:BingrillofrdcFgrill}),
this is equivalent to 
\[
\{V^{-}:V\in\V_{\xi}(x),x\in A\}\subset\FF,
\]
and $\{V^{-}:V\in\V_{\xi}(x),x\in A\}$ has the finite intersection
property since $X\in\cap_{V\in\V_{\xi}(x),x\in A}V^{-}$. Hence $\bigvee_{x\in A}\{V^{-}:V\in\V_{\xi}(x)\}^{\uparrow}$
is a proper filter and is the vicinity filter of $A$ for $\lim_{lK}$.
\end{proof}

\subsubsection{A word on nets}

Recall that a net $\X=\left\langle x_{\alpha}\right\rangle {}_{\alpha\in\Lambda}$
on a set $X$ is a map $\X:\Lambda\to X$ where $\Lambda$ is directed
set. If $X$ is a topological space we say that $\left\langle x_{\alpha}\right\rangle {}_{\alpha\in\Lambda}$
\emph{converges to }$t$ if for every $V\in\N_{X}(t)$ there is $\alpha\in\Lambda$
such that $x_{\beta}\in V$ for every $\beta\in\Lambda$ with $\beta\geq\alpha$,
that is, if 
\[
\F_{\X}\geq\N_{X}(t)
\]
where $\F_{\X}$ is the \emph{associated filter }given by
\[
\F_{\X}:=\{\{x_{\beta}:\beta\geq_{\Lambda}\alpha\}:\alpha\in\Lambda\}^{\uparrow_{X}}.
\]

Hence we generally define the convergence of a net $\X$ on a convergence
space $(X,\xi)$ by
\[
\lm_{\xi}\X=\lm_{\xi}\F_{\X}.
\]

Recall as well that a subset $R$ of a directed set $\Lambda$ is
called \emph{residual }if there is $\alpha\in\Lambda$ with $\{\beta\geq_{\Lambda}\alpha\}\subset R$
and $C\subset\Lambda$ is \emph{cofinal }if $\Lambda\setminus C$
is not residual, that is, if for each $\alpha\in\Lambda$, there is
$\beta\in C$ with $\beta\geq\alpha$. Let $\mathfrak{R}(\Lambda)$
denote the set of residual subsets of $\Lambda$ and $\mathfrak{C}(\Lambda)$
denote the set of cofinal subsets of $\Lambda$. It is clear that
given a net $\A=\left\langle A_{\alpha}\right\rangle _{\alpha\in\Lambda}$
on $\C_{c(\lambda)}$,
\begin{eqnarray}
R\in\mathfrak{R}(\Lambda) & \then & \{A_{\alpha}:\alpha\in R\}\in\FF_{\A}\label{eq:residual}\\
C\in\mathfrak{C}(\Lambda) & \then & \{A_{\alpha}:\alpha\in C\}\in\FF_{\A}^{\#}\label{eq:cofinal}
\end{eqnarray}
and the converses are true for a one-to-one (or finite-to-one) net.

It is worth noting that convergence of nets for hyperspace structures
have often been defined or characterized directly for nets in terms
seemingly different. For instance, the seminal book \cite{Beer},
defines the following given a net $\A$ of subsets of a Hausdorff
space $X$:
\global\long\def\Li{\operatorname{Li}}%
\global\long\def\Ls{\operatorname{Ls}}%
\[
\Li\A=\{x\in X:\forall V\in\N_{X}(x)\exists R\in\mathfrak{R}(\Lambda),V\in\{A_{\alpha}:\alpha\in R\}^{\#}\}
\]
and 
\[
\Ls\A=\{x\in X:\forall V\in\N_{X}(x)\exists C\in\mathfrak{C}(\Lambda),V\in\{A_{\alpha}:\alpha\in C\}^{\#}\}
\]
 and says that the net $\A$ \emph{Kuratowski-Painlevé converges to
$C$ }if $\Li\A=\Ls\A=C$.
\begin{prop}
If $\A=\left\langle A_{\alpha}\right\rangle _{\alpha\in\Lambda}$
is a net of subsets of a topological space $X$ then 
\[
\Li\A=\adh(\rdc\FF_{\A}^{\#})
\]
and 
\[
\Ls\A=\adh(\rdc\FF_{\A}).
\]
\end{prop}

\begin{proof}
Note that in a topological space $x\in\adh(\rdc\FF_{\A}^{\#})$ if
and only if $\N(x)\#\rdc\FF_{\A}^{\#}$, equivalently, in view of
(\ref{eq:BingrillofrdcFgrill}),
\[
\{V^{-}:V\in\N(x)\}\subset\FF_{\A},
\]
equivalently, $x\in\Li\A$. 

Similarly, $x\in\adh\rdc\FF_{\A}$ if and only if $\N(x)\#\rdc\FF_{\A}$,
which, in view of (\ref{eq:BingrillofrdcF}), is equivalent to $\{V^{-}:V\in\N(x)\}\subset\FF_{A}^{\#}$,
equivalently, $x\in\Ls\A$.
\end{proof}
Hence, since $\adh\rdc\FF^{\#}\subset\adh\rdc\FF$ for every filter
$\FF$ on the powerset of $X$, $\text{\ensuremath{\A}}$ Kuratowski-Painlevé
converges to $C$ if and only if 
\[
\adh(\rdc\FF_{\A})\subset C\subset\adh(\rdc\FF_{\A}^{\#}).
\]
In other words, if we consider $C$ a closed subset and $\A$ of net
of closed subsets of $X$, then $\A$ Kuratowski-Painlevé converges
to $C$ if and only $C\in\lim_{uK}\FF_{\A}\cap\lim_{lK}\FF_{A}=\lim_{K}\FF_{A}$.

\section{(pre)Convergence approach spaces}

A \emph{preconvergence approach space }$(X,\lambda)$ is given by
a function $\lambda:\mathbb{F}X\to[0,\infty]^{X}$ satisfying 
\begin{equation}
\F\leq\G\then\lambda(\F)\geq\lambda(\G)\label{eq:CAPmonotone}
\end{equation}
for the pointwise order of $[0,\infty]^{X}$ induced by the usual
order on $[0,\infty]$. The interpretation is that $\lambda(\F)(x)$
measures how well $\F$ converges to $x$, with $\lambda(\F)(x)=0$
indicating full convergence while $\lambda(\F)(x)=\infty$ indicates
no convergence.

If moreover 
\begin{equation}
\tag{precentered}\lambda(\{x\}^{\uparrow})(x)<\infty,\label{eq:CAPprecentered}
\end{equation}
for every $x\in X$, we say that $(X,\lambda)$ is \emph{precentered}.
It is called \emph{centered }if 
\begin{equation}
\tag{centered}\lambda(\{x\}^{\uparrow})(x)=0,\label{eq:CAPfullycentered}
\end{equation}
for all $x\in X$. A centered preconvergence approach space is called
a \emph{convergence approach space}. 

Additionally, 
\begin{equation}
\lambda(\F\cap\G)=\lambda(\F)\vee\lambda(\G),\label{eq:capfinitedepth}
\end{equation}
for every $\F,\G\in\mathbb{F}X$ is usually required of convergence
approach spaces (e.g., \cite{lowe88,Lowen88,towers}), though we will
treat it as an additional property, called \emph{finite depth}. 

A \emph{contraction} is a function $f:(X,\lambda_{X})\to(Y,\lambda_{Y})$
satisfying 
\[
\lambda_{Y}(f[\F])\circ f\leq\lambda_{X}(\F)
\]
 for the pointwise order induced by that of $[0,\infty]$, for every
$\F\in\mathbb{F}X$. Let $\cont(X,Y)$ denote the set of all such
morphisms. Let $\Cap$ denote the resulting category with these morphisms
and convergence approach spaces as objects. Let $\Cap_{*}$ denote
the larger category of preconvergence approach spaces and contractions
and $\Cap_{c}$ denotes its full subcategory of precentered preconvergence
approach spaces.

Note that $\Cap$ is a topological category (\footnote{$\Cap$ is also cartesian-closed e.g., \cite{lowe88}.}),
where initial structures are as follows: If $f_{i}:X\to(Y_{i},\lambda_{i})$
then the $\Cap$-initial structure on $X$ is given by 
\begin{equation}
\lambda_{X}(\F)(x)=\bigvee_{i\in I}\lambda_{i}(f_{i}[\F])(f_{i}(x)).\label{eq:initialCAP}
\end{equation}

\begin{rem}
Note that this is also the initial structure in $\Cap_{*}$, which
is thus topological. On the other hand, $\Cap_{c}$ is not topological.
For instance, suppose $Y_{i}=\{y\}$ for all $i\in\mathbb{N}$ and
$\lambda_{i}(y)(y)=i$. Then $(Y_{i},\lambda_{i})$ is precentered
for each $i$ but if $f_{i}:X\to(Y_{i},\lambda_{i})$ then the structure
given by (\ref{eq:initialCAP}) on $X$ is not.
\end{rem}

If $(X,\lambda)$ is a preconvergence approach space and $\F\subset\mathbb{P}X$
then the \emph{adherence function of $\F$ }is the function $\adh_{\lambda}\F:X\to V$
(or $\adh\F$ if $\lambda$ is unambiguous) defined by
\[
\adh_{\lambda}\F(\cdot)=\bigwedge_{\G\#\F}\lambda(\G)(\cdot).
\]
If $\F\in\mathbb{F}X$ then 
\[
\adh_{\lambda}\F(\cdot)=\bigwedge_{\G\#\F}\lambda(\G)(\cdot)=\bigwedge_{\U\in\beta(\F)}\lambda\U(\cdot)
\]
and in particular, if $A\subset X$,
\[
\adh_{\lambda}A(\cdot):=\adh_{\lambda}\{A\}^{\uparrow}(\cdot)=\bigwedge_{A\in\F^{\#}}\lambda\F(\cdot)=\bigwedge_{\U\in\beta A}\lambda\U(\cdot).
\]

If $A\subset X$, let $\theta_{A}(x)=\begin{cases}
0 & x\in A\\
\infty & x\notin A
\end{cases}$ denote the \emph{indicator function of $A$.}

Given a preconvergence approach space $(X,\lambda)$, $A\subset X$,
and $\epsilon\in[0,\infty]$, we will use the notation
\[
A^{(\epsilon)}:=\adh^{-1}[0,\epsilon]=\left\{ x\in X:\adh A(x)\leq\epsilon\right\} .
\]

\begin{lem}
\label{lem:centered-1} The following are equivalent for a preconvergence
approach space $(X,\lambda)$:
\begin{enumerate}
\item $(X,\lambda)$ is centered, hence a convergence approach space;
\item $\adh A\leq\theta_{A}$ for all $A\subset X$;
\item $A\subset A^{(\epsilon)}$ for every $A\subset X$ and every $\epsilon\in[0,\infty)$;
\item $A\subset A^{(0)}$ for every $A\subset X$.
\end{enumerate}
\end{lem}

\begin{proof}
$(1)\then(2)$: If $\lambda$ is centered and $A\subset X$ than $A\in\{x\}^{\uparrow}$
for all $x\in A$ and thus $\adh A(x)=0$, equivalently, $\adh A\leq\theta_{A}$.
$(2)\then(3)$: If $\adh A\leq\theta_{A}$ and $\epsilon\in[0,\infty)$,
then $\theta_{A}^{-1}([0,\epsilon])\subset\adh^{-1}A([0,\epsilon])$,
that is, $A\subset A^{(\epsilon)}$. $(3)\then(4)$ is clear and $(4)\then(1)$
because given $x\in X,$ letting $A=\{x\}$ in $(4)$ yields $\{x\}\subset\{x\}^{(0)}$
so that $\adh\{x\}(x)=0$, equivalently, $\lambda(\{x\}^{\uparrow})(x)=0$.
\end{proof}
Similarly,
\begin{lem}
\label{lem:precentered} Let $(X,\lambda)$ be a a preconvergence
approach space. Then $(X,\lambda)$ is precentered if and only if
\[
A\subset\bigcup_{\epsilon\in[0,\infty)}A^{(\epsilon)}
\]
 for every $A\subset X$.
\end{lem}

\begin{proof}
If $\lambda$ is precentered and $A\subset X$ then for every $x\in A$,
$\epsilon_{x}:=\lambda(\{x\}^{\uparrow})(x)$ is finite and $x\in A^{(\epsilon_{x})}$.
Conversely, if $A\subset\bigcup_{\epsilon\in[0,\infty)}A^{(\epsilon)}$
for every $A$, then in particular, given $x\in X$, for $A=\{x\}$
we have $\{x\}\subset\bigcup_{\epsilon\in[0,\infty)}\{x\}^{(\epsilon)}$
so that there is $\epsilon\in[0,\infty)$ with $x\in\{x\}^{(\epsilon)}$,
equivalently, $\lambda(\{x\}^{\uparrow})(x)\leq\epsilon<\infty$ and
thus $\lambda$ is precentered.
\end{proof}
A convergence approach space $(X,\lambda)$ is a called a \emph{pre-approach
space }(e.g., \cite{Lowen97}) if for any $\mathbb{D}\subset\mathbb{F}X$,
\begin{equation}
\lambda(\bigcap_{\D\in\mathbb{D}}\D)(\cdot)=\bigvee_{\D\in\mathbb{D}}\lambda\D(\cdot).\tag{PrAp}\label{eq:preAP}
\end{equation}

Note that pre-approach spaces are determined by adherence functions
of principal filters. Namely, in a pre-approach space, e.g., \cite{mynardmeasureCAP},
\begin{equation}
\lambda(\F)(\cdot)=\bigvee_{A\in\F^{\#}}\adh_{\lambda}A(\cdot),\label{eq:limPRAP}
\end{equation}
and
\begin{equation}
\adh_{\lambda}\F(\cdot)=\bigvee_{F\in\F}\adh_{\lambda}F(\cdot).\label{eq:adhPRAP}
\end{equation}

If (\ref{eq:preAP}) is only true for subsets $\mathbb{D}$ of $\mathbb{U}X$
of the form $\beta(\F)$, that is, if 
\begin{equation}
\tag{PsAp}\lambda(\F)(\cdot)=\bigvee_{\U\in\beta(\F)}\lambda\U(\cdot)\label{eq:psap}
\end{equation}
for every $\F\in\mathbb{F}X$, then $(X,\lambda)$ is a called a \emph{pseudo-approach
space }(e.g.,\cite{lowe88}).

A pre-approach space $(X,\lambda)$ is an \emph{approach space }if
it satisfies the additional condition that 
\begin{equation}
\adh_{\lambda}A\leq\adh_{\lambda}A^{(\epsilon)}+\epsilon\label{eq:diagadh}
\end{equation}
for each $A\subset X$ and each $\epsilon\in[0,\infty]$.
\begin{rem}
Note that in an approach space, $\adh A(\cdot)$ has traditionally
been denoted $\delta(\cdot,A)$ and thought of as the distance from
$A$, valued in $[0,\infty]$, so that (\ref{eq:diagadh}) is usually
written
\[
\delta(\cdot,A)\leq\delta(\cdot,A^{(\epsilon)})+\epsilon.
\]
\end{rem}

A \emph{(pre)convergence tower} on $X$ \cite{towers} is a collection
of (pre)convergences $\{\sigma_{\epsilon}:\epsilon\in[0,\infty]\}$
on $X$ indexed by $[0,\infty]$ satisfying:
\begin{eqnarray*}
\epsilon\leq\gamma & \then & \sigma_{\gamma}\leq\sigma_{\epsilon}\\
\sigma_{\infty} &  & \text{antidiscrete}\\
\sigma_{\epsilon} & = & \bigvee_{\epsilon<\gamma}\sigma_{\gamma}.
\end{eqnarray*}

A morphism between (pre)convergence towers $(X,\{\sigma_{\epsilon}:\epsilon\in[0,\infty]\})$
and $(Y,\{\tau_{\epsilon}:\epsilon\in[0,\infty]\})$ is a function
$f:X\to Y$ such that $f\in\cont(\sigma_{\epsilon},\tau_{\epsilon})$
for all $\epsilon\in[0,\infty]$. It turns out that this category
is equivalent to the category $\Cap_{*}$ \cite{towers}. More specifically,
a $\Cap_{*}$-object $(X,\lambda)$ corresponds to the preconvergence
tower $\{\lambda_{\epsilon}:\epsilon\in[0,\infty]\}$ on $X$ given
by 
\begin{equation}
x\in\lm_{\lambda_{\epsilon}}\F\iff\lambda(\F)(x)\leq\epsilon,\label{eq:CAPtotower}
\end{equation}
while a (pre)convergence tower $\{\sigma_{\epsilon}:\epsilon\in[0,\infty]\}$
on $X$ corresponds to a $\Cap_{*}$-object $(X,\lambda_{\sigma})$
given by
\begin{equation}
\lambda_{\sigma}\F(x)=\inf\left\{ \epsilon\in[0,\infty]:x\in\lm_{\sigma_{\epsilon}}\F\right\} ,\label{eq:towertoCAP}
\end{equation}
and this defines an equivalence of categories, which restricts to
an equivalence of categories between convergence towers and $\Cap$.
In view of \cite[Proposition 10]{towers} this restricts further to
equivalences of categories between towers of pseudotopologies and
$\mathsf{Psap}$ and between towers of pretopologies and $\mathsf{Prap}$. 

\begin{defn}
Let $\G:X\to\mathbb{F}Y$ and $\F\in\mathbb{F}X$. Then, the \emph{contour
}
\[
\G(\F)=\bigcup_{F\in\F}\bigcap_{t\in F}\G(t)
\]
is a filter on $Y$ \cite{DM.book} .

With these notations:
\end{defn}

\begin{prop}
\cite[Proposition 11 and Corollary 14]{towers}\label{prop:diagonaltower}
The category $\mathsf{Ap}$ is equivalent to the category of towers
of pretopologies $(\sigma_{\epsilon})_{\epsilon\in[0,\infty]}$ satisfying
\begin{equation}
\left(\forall y\in X,y\in\lm_{\sigma_{\epsilon}}\mathcal{S}(y)\right)\text{ and }x\in\lm_{\sigma_{\gamma}}\F\then x\in\lm_{\sigma_{\epsilon+\gamma}}\mathcal{S}(\F),\label{eq:diagonaltower}
\end{equation}
whenever $\mathcal{S}:X\to\mathbb{F}X$ and $\F\in\mathbb{F}X$.
\end{prop}

\begin{prop}
\label{prop:twodiagonalaxioms} Let $(X,\lambda)$ be a convergence
approach space. If for every $\G:X\to\mathbb{F}X$, every $\F\in\mathbb{F}X$
and every $x_{0}\in X$
\begin{equation}
\lambda(\G(\F))(x_{0})\leq\lambda(\F)(x_{0})+\bigvee_{t\in X}\lambda(\G(t))(t),\label{eq:diagfilters}
\end{equation}
 then (\ref{eq:diagadh}) for every $A\subset X$ and every $\epsilon\in[0,\infty]$.
The converse is true if $(X,\lambda)$ is a pre-approach space. 
\end{prop}

\begin{proof}
Assume (\ref{eq:diagfilters}) and let $A\subset X$, $\epsilon\in[0,\infty]$
and $x_{0}\in X$. For every $t\in A^{(\epsilon)}$ there is $\G(t)\in\mathbb{U}X$
with $A\in\G(t)$ and $\lambda(\G(t))(t)\leq\epsilon$. For every
$\F\in\mathbb{U}X$ with $A^{(\epsilon)}\in\F$, $A\in\G(\F)$ and
thus
\[
\adh A(x_{0})\leq\lambda(\G(\F))(x_{0})\leq\lambda(\F)(x_{0})+\bigvee_{t\in A^{(\epsilon)}}\lambda(\G(t))(t)\leq\lambda(\F)(x_{0})+\epsilon,
\]
so that 
\[
\adh A(x_{0})\leq\bigwedge_{A^{(\epsilon)}\in\F\in\mathbb{U}X}\lambda\F(x_{0})+\epsilon\leq\adh A^{(\epsilon)}(x_{0})+\epsilon.
\]

Assume (\ref{eq:diagadh}) for all $A\subset X$ and every $\epsilon$
and let $\G:X\to\mathbb{F}X$, $\F\in\mathbb{F}X$ and $x_{0}\in X$.
To show (\ref{eq:diagfilters}), we may assume $\epsilon:=\bigvee_{t\in X}\lambda(\G(t))(t)<\infty$.
As $(X,\lambda)$ is pre-approach, 
\[
\lambda(\G(\F))(x_{0})=\bigvee_{A\#\G(\F)}\adh A(x_{0}).
\]

For each $A\in(\G(\F))^{\#}=\G^{\#}(\F^{\#})$ (See e.g., \cite[Prop. VI.1.4]{DM.book}),
there is $H\in\F^{\#}$ and, for each $t\in H$, there is $S_{t}\in\G(t)^{\#}$
such that $\bigcup_{t\in H}S_{t}\subset A$. By (\ref{eq:diagadh}),
\[
\adh A(x_{0})\leq\adh(\bigcup_{t\in H}S_{t})(x_{0})\leq\adh(\bigcup_{t\in H}S_{t})^{(\epsilon)}(x_{0})+\epsilon
\]
and $\adh(\bigcup_{t\in H}S_{t})^{(\epsilon)}\leq\lambda\F(x_{0})$
because $(\bigcup_{t\in H}S_{t})^{(\epsilon)}\#\F$. Indeed, $H\subset(\bigcup_{t\in H}S_{t})^{(\epsilon)}$
because for each $t\in H$, $\lambda(\G(t))(t)\leq\epsilon$ so that
$t\in(S_{t})^{(\epsilon)}\subset(\bigcup_{t\in H}S_{t})^{(\epsilon)}$.
As $H\#\F$, the conclusion follows.

Therefore, $\adh A(x_{0})\leq\lambda\F(x_{0})+\epsilon$ and thus
\[
\lambda(\G(\F))(x_{0})=\bigvee_{A\#\G(\F)}\adh A(x_{0})\leq\lambda\F(x_{0})+\epsilon=\lambda(\F)(x_{0})+\bigvee_{t\in A^{(\epsilon)}}\lambda(\G(t))(t).
\]
\end{proof}
\begin{rem}
\label{rem:pointofdiag}Note that the proof shows that for a given
$x_{0}\in X$, (\ref{eq:diagfilters}) for every $\F\in\mathbb{F}X$
and $\G:X\to\mathbb{F}X$ implies 
\begin{equation}
\adh_{\lambda}A(x_{0})\leq\adh_{\lambda}A^{(\epsilon)}(x_{0})+\epsilon\label{eq:adhdiagatx0}
\end{equation}
for every $A\subset X$ and every $\epsilon\in[0,\infty]$, and conversely
if $(X,\lambda)$ is a pre-approach space. If $\F\in\mathbb{F}X$,
let us call $x_{0}\in X$ \emph{a point of }$\F$-\emph{diagonality
}of $(X,\lambda)$ if (\ref{eq:diagfilters}) for every $\G:X\to\mathbb{F}X$
and a point of \emph{filter-diagonality} of $(X,\lambda)$ if it is
a point of $\F$-diagonality for every $\F\in\mathbb{F}X$. We say
that $x_{0}$ is a point of \emph{adherence-diagonality} if (\ref{eq:adhdiagatx0})
for every $A\subset X$ and every $\epsilon\in[0,\infty]$. According
to Proposition \ref{prop:twodiagonalaxioms}, filter-diagonality and
adherence-diagonality are equivalent among pre-approach spaces.
\end{rem}

If a pre-approach space satisfies the following stronger version of
(\ref{eq:diagfilters}), it is called a \emph{non-Archimedean approach
space} (and is of course an approach space):
\begin{equation}
\lambda(\G(\F))(x_{0})\leq\lambda(\F)(x_{0})\vee\bigvee_{t\in X}\lambda(\G(t))(t),\label{eq:nonarchimedean}
\end{equation}
for every $\G:X\to\mathbb{F}X$, every $\F\in\mathbb{F}X$ and every
$x_{0}\in X$.

Let $\psap$, $\prap$ and $\ap$ denote the full subcategories of
$\Cap$ of pseudo-approach, pre-approach and approach spaces respectively.

Recall as well (e.g., \cite{Lowen88}) that $c:\Cap\rightarrow\conv$
and $r:\Cap\rightarrow\conv$ defined on objects by 
\begin{eqnarray}
x & \in & \lim\nolimits_{c(\lambda)}\mathcal{F\Longleftrightarrow\lambda F}(x)=0\label{eq:ccorefl}\\
x & \in & \lim\nolimits_{r(\lambda)}\mathcal{F\Longleftrightarrow\lambda F}(x)<\infty\label{eq:rrefl}
\end{eqnarray}
extend to a concrete coreflector and a concrete reflector respectively,
right and left adjoint to the inclusion functor $i:\conv\to\Cap$
defined on objects by $i(X,\xi)=(X,\lambda_{\xi})$ where $\lambda_{\xi}(\F)(x)=\begin{cases}
0 & x\in\lim_{\xi}\F\\
\infty & x\notin\lim_{\xi}\F
\end{cases}$. The functors $c$ and $r$ restrict to coreflectors and reflectors
from $\psap$ to $\pstop$, $\mathbf{\prap}$ to $\prtop$ and $\ap$
to $\top$. 
\begin{rem}
Note that the assumption that $(X,\lambda)$ be a pre-approach space
cannot be dropped in the converse part of Proposition \ref{prop:twodiagonalaxioms}.
Indeed, for a convergence space $(X,\xi)$ the corresponding convergence
approach space $(X,\lambda_{\xi})$ satisfies (\ref{eq:diagadh})
if and only if $\adh_{\xi}(\adh_{\xi}A)=\adh_{\xi}A$ for every $A\subset X$,
equivalently, $\S_{0}\xi=\T\xi$. On the other hand, $(X,\lambda_{\xi})$
satisfies (\ref{eq:diagfilters}) if and only if $\xi$ is \emph{diagonal},
that is, $\lim_{\xi}\F\subset\lim_{\xi}\G(\F)$ for every $\F\in\mathbb{F}X$
and $\G:X\to\mathbb{F}X$ satisfying $x\in\lim_{\xi}\G(x)$ for all
$x\in X$. There are known examples of non-diagonal convergence spaces
$(X,\xi)$ satisfying $\S_{0}\xi=\T\xi$, e.g., \cite[Example VI.2.14]{DM.book}.
\end{rem}

\begin{rem}
Note that (\ref{eq:ccorefl}) and (\ref{eq:rrefl}) extend to a coreflector
$c:\Cap_{*}\to\pconv$ and a reflector $r:\Cap_{*}\to\pconv$. Moreover,
a preconvergence approach space $(X,\lambda)$ is precentered if and
only if $r(\lambda)$ is a convergence and centered if and only if
$c(\lambda)$ is a convergence. Hence $r:\Cap_{c}\to\conv$.
\end{rem}

Let $\ominus$ denote the truncated difference on $[0,\infty]$ defined
by 
\[
x\ominus y=\begin{cases}
(x-y)\vee0 & x,y\in[0,\infty)\\
\infty & x=\infty,y\neq\infty\\
0 & y=\infty
\end{cases}.
\]
Recall e.g., \cite{AP.book,indextheory}, that $[0,\infty]$ can be
endowed with a canonical convergence approach structure $\lambda_{V}$
given by
\begin{eqnarray*}
\lambda_{V}(\F)(v) & = & v\ominus(\bigwedge_{A\in\F^{\#}}\bigvee A)=\bigvee_{A\in\F^{\#}}v\ominus\bigvee A=\bigvee_{A\in\F^{\#}}\adh_{\lambda_{V}}A(v),
\end{eqnarray*}
which turns out to be an approach space. 

If $V=([0,\infty],\lambda_{V})$ and $(X,\lambda)$ is a convergence
approach space and $A\subset X$, then (See \cite{myn.APinCAP})
\[
\adh_{\lambda}A\in\cont(X,V)\iff\forall\epsilon\in[0,\infty]\;(\adh_{\lambda}A\leq\adh_{\lambda}A^{(\epsilon)}+\epsilon).
\]
Moreover, 
\begin{equation}
\theta_{A}\in\cont(X,V)\iff A\in\C_{r(\lambda)},\label{eq:closed}
\end{equation}
in which case $\adh_{\lambda}A=\theta_{A}$. 

Recall (\footnote{Though this is formulated for approach spaces, the proof only depends
on the fact that this takes place in a pre-approach space.}):
\begin{prop}
\label{prop:adhpreap}\cite[Prop. I.2.67(2)]{indextheory} If $(X,\lambda)$
is a pre-approach space, $x\in X$ and $\F\in\mathbb{F}X$, then $\adh\F(x)=\alpha$
if and only if there is $\U\in\beta\F$ with $\lambda(\U)(x)=\alpha$. 
\end{prop}

\begin{cor}
\label{cor:toweradh} Let $(\tau_{\epsilon})_{\epsilon\in[0,\infty]}$
denote the tower of pretopologies associated with a pre-approach space
$(X,\lambda)$. Then 
\[
\{\adh_{\lambda}\FF\leq\epsilon\}=\adh_{\tau_{\epsilon}}\FF.
\]
\end{cor}

\begin{prop}
\label{prop:adhpreapcorefl}Let $(X,\lambda)$ be a preconvergence
approach space and $\F\in\mathbb{F}X$, then
\[
\adh_{r(\lambda)}\F=\{\adh_{\lambda}\F<\infty\},
\]
and
\[
\adh_{c(\lambda)}\F\subset\{\adh_{\lambda}\F=0\}.
\]
Moreover, 
\[
\adh_{c(\lambda)}\F=\{\adh_{\lambda}\F=0\}
\]
if $(X,\lambda)$ is a pre-approach space. 

In particular, if $A\subset X$ then
\[
\adh_{r(\lambda)}A=\bigcup_{\epsilon<\infty}A^{(\epsilon)},
\]
and if $(X,\lambda)$ is moreover a pre-approach space, $\adh_{c(\lambda)}A=A^{(0)}$.
\end{prop}

\begin{proof}
Note that $x\in\adh_{r(\lambda)}\F$ if and only if there is a filter
$\G$ with $\G\#\F$ and $x\in\lm_{r(\lambda)}\G$, equivalently,
$\lambda(\G)(x)<\infty$, so that $\adh_{\lambda}\F(x)<\infty$. conversely,
if $\adh_{\lambda}\F(x)=\bigwedge_{\G\#\F}\lambda(\G)(x)$ is finite,
there is a filter $\G$ with $\G\#\F$ and $\lambda(\G)(x)<\infty$,
equivalently, $x\in\lm_{r(\lambda)}\G$.

Similarly, $x\in\adh_{c(\lambda)}\F$ if and only if there is a filter
$\G$ with $\G\#\F$ and $x\in\lm_{c(\lambda)}\G$, equivalently,
$\lambda(\G)(x)=0$, so that $\adh_{c(\lambda)}\F\subset\{\adh_{\lambda}\F=0\}$.
In view of Proposition \ref{prop:adhpreap}, the reverse inclusion
is true in a pre-approach space.
\end{proof}
\begin{example}[A subset $A$ of convergence approach space $(X,\lambda)$ where $A^{(0)}\not\subset\adh_{c(\lambda)}A$]

Let $X$ be an infinite set and let $x_{0}\in X$. Let $(\U_{n})_{n\in\mathbb{N}}$
be a sequence of distinct free ultrafilters on $X\setminus\{x_{0}\}$
and let $\lambda:\mathbb{F}X\to V^{X}$ be given by 
\[
\lambda(\F)(x)=\begin{cases}
\frac{1}{n} & \F=\U_{n},x=x_{0}\\
\infty & \F\in\{\U_{n}:n\in\mathbb{N}\}x\neq x_{0}\\
\infty & \F\notin\{\U_{n}:n\in\mathbb{N}\}\cup\{\{t\}^{\uparrow}:t\in X\}\\
0 & \F=\{x\}^{\uparrow}
\end{cases}.
\]

This is a convergence approach space. Moreover, if $A=X\setminus\{x_{0}\}$
then $x_{0}\notin A=\adh_{c(\lambda)}A$ but $x_{0}\in A^{(0)}$ because
\[
\adh_{\lambda}A(x_{0})=\bigwedge_{\U\in\beta A}\lambda\U(x_{0})\leq\bigwedge_{n\in\mathbb{N}}\lambda\U_{n}(x_{0})=0.
\]
\end{example}

\begin{cor}
A subset $A$ of a preconvergence approach space $(X,\lambda)$ is
$r(\lambda)$-closed if and only if $A^{(\epsilon)}=A$ for every
$\epsilon\in[0,\infty)$, while in a pre-approach space, $A$ is $c(\lambda)$-closed
if and only if $A=A^{(0)}$.
\end{cor}

In particular, every $r(\lambda)$-closed set is also $c(\lambda)$-closed,
that is, $\C_{r(\lambda)}\subset\C_{c(\lambda)}$.

If $f:X\to[0,\infty]$, $v\in[0,\infty]$ and $\square$ is one of
$<,\leq,=,\geq,>$, we use the shorthand 
\[
\{f\square v\}:=\{x\in X:f(x)\square v\}.
\]

\begin{lem}
\label{lem:alphaofadhlesseps} If $(X,\lambda)$ is an approach space,
$\F\in\mathbb{F}X$, $\epsilon,\alpha\in[0,\infty]$ then
\[
\{\adh\F<\epsilon\}^{(\alpha)}\subset\{\adh\F\leq\epsilon\}^{(\alpha)}\subset\{\adh\F\leq\epsilon+\alpha\}.
\]
\end{lem}

\begin{proof}
Suppose $x\in\{\adh\F\leq\epsilon\}^{(\alpha)}$. In view of Proposition
\ref{prop:adhpreap}, there is an ultrafilter $\U$ with $\{\adh\F\leq\epsilon\}\in\U$
and $\lambda(\U)(x)\leq\alpha$. For every $t\in\{\adh\F\leq\epsilon\}$,
there is $\W_{t}\in\beta(\F)$ with $\lambda(\W_{t})(t)\leq\epsilon$.
By (\ref{eq:diagfilters}),
\[
\lambda(\W_{t}(\U))(x)\leq\lambda(\U)(x)+\bigvee_{t\in\{\adh\F<\epsilon\}}\lambda(\W_{t})(t)\leq\alpha+\epsilon.
\]
Moreover, $\W_{t}(\U)\#\F$ because $\W_{t}\in\beta(\F)$ for every
$t$ and thus $x\in\{\adh\F\leq\epsilon+\alpha\}$. Taking $\F=\{A\}^{\uparrow}$
yields the particular case.
\end{proof}
\begin{rem}
\label{rem:clambdaclusre} In view of Proposition \ref{prop:adhpreapcorefl},
if $(X,\lambda)$ is a pre-approach space,
\[
\adh_{c(\lambda)}\left\{ \adh_{\lambda}\F<\epsilon\right\} =\left\{ \adh_{\lambda}\F<\epsilon\right\} ^{(0)}\subset\left\{ \adh_{\lambda}\F\leq\epsilon\right\} 
\]
because $\left\{ \adh_{\lambda}\F\leq\epsilon\right\} =\bigcap_{F\in\F}F^{(\epsilon)}$,
which is $c(\lambda)$-closed by Lemma \ref{lem:alphaofadhlesseps}.
The reverse inclusion is false:
\end{rem}

\begin{example}
Let $X$ be an infinite set and let $x_{0}\in X$. Let $\mathbb{F}^{*}X$
denote the set of free filters on $X$. Let $\lambda:\mathbb{F}X\to V^{X}$
be given by 
\[
\lambda(\F)(x)=\begin{cases}
1 & \F\in\mathbb{F}^{*}X,x=x_{0}\\
\infty & \F\in\mathbb{F}^{*}X,x\neq x_{0}\\
\infty & \F\notin\mathbb{F}^{*}X\cup(\mathbb{U}X\cap\mathbb{F}_{0}X)\\
0 & \F=\{x\}^{\uparrow}
\end{cases}.
\]
This is easily seen to be a $\prap$-structure on $X$. If $A$ is
an infinite subset of $X$ with $x_{0}\notin A$ then 
\[
\{x\in X:\adh A(x)<1\}^{(0)}=A^{(0)}=A
\]
and $\{x\in X:\adh A(x)\leq1\}=A\cup\{x_{0}\}$.
\end{example}

\begin{rem}
\label{rem:netsinCAP} If $(X,\lambda)$ is a (pre)convergence approach
space we define the convergence of nets as we did for topological
or convergence spaces. Namely, if $\X$ is a net on $X$ then
\[
\lambda\X(x):=\lambda\F_{\X}(x).
\]
\end{rem}

\section{Kuratowski Convergence In Cap Spaces}

\global\long\def\UU{\mathfrak{U}}%
We will consider the operation $\oslash$ defined for $x\in[0,\infty)$
and $y\in[0,\infty]$ by 
\[
x\oslash y=\begin{cases}
x/y & y\notin\{0,\infty\},\\
0 & y=\infty\\
\infty & y=0
\end{cases}.
\]

Note that 
\begin{equation}
1\oslash\bigwedge_{a\in A}a=\bigvee_{a\in A}1\oslash a,\label{eq:oslahofinf}
\end{equation}
for every $A\subset[0,\infty]$, and 
\begin{equation}
1\oslash x\leq y\iff1\oslash y\leq x\label{eq:divideineq}
\end{equation}
for every $x,y\in[0,\infty]$.

Given a convergence approach space $(X,\lambda),$ let $\C_{r(\lambda)}$,
respectively $\C_{c(\lambda)}$, denote the set of all $r(\lambda)$-closed
, respectively $c(\lambda)$-closed subsets of $X$.

Let $\FF\in\mathbb{F}(\C_{c(\lambda)})$ and $A\in\C_{c(\lambda)}$
and define 
\begin{equation}
\lambda_{uK}\FF(A)=\bigvee_{x\notin A}1\oslash\adh_{\lambda}(\rdc\FF)(x)=1\oslash\bigwedge_{x\notin A}\adh_{\lambda}(\rdc\FF)(x)\label{eq:lambdauK}
\end{equation}
\begin{equation}
\lambda_{lK}\FF(A)=\bigvee_{x\in A}\adh_{\lambda}(\rdc\FF^{\#})(x),\label{eq:lambdalK}
\end{equation}
and 
\[
\lambda_{K}\FF(A)=\lambda_{uK}\FF(A)\vee\lambda_{lK}\FF(A).
\]

Given an $\epsilon\in[0,\infty]$, the filter $\FF$ is said to be
$\epsilon$-\emph{upper Kuratowski convergent} (respectively $\epsilon$-\emph{lower
Kuratowski convergent}, respectively $\epsilon$-\emph{Kuratowski
convergent}) to $A$ if $\lambda_{uK}\FF(A)\leq\epsilon$ (respectively
$\lambda_{lK}\FF(A)\leq\epsilon$, respectively $\lambda_{K}\FF(A)\leq\epsilon$),
in which case we write $A\in\lm_{uK_{\epsilon}}\FF$, respectively
$A\in\lm_{lK_{\epsilon}}\FF$, respectively $A\in\lm_{K_{\epsilon}}\FF$.

Note in particular that if $\UU\in\mathbb{U}(\C_{c(\lambda)})$ then
\[
\lambda_{lK}\UU(A)=\bigvee_{x\in A}\adh_{\lambda}(\rdc\UU)(x),
\]
 and 
\[
\lambda_{K}\UU(A)=\bigvee_{x\in A}\adh_{\lambda}(\rdc\UU)(x)\vee\bigvee_{x\notin A}1\oslash\adh_{\lambda}(\rdc\UU)(x).
\]

It follows immediately from the definitions that
\begin{lem}
\label{lem:uKantitoneonlimitpoints}If $\FF\in\mathbb{F}\C_{c(\lambda)}$,
$A,B\in\C_{c(\lambda)}$ and $A\subset B$ then
\[
\lambda_{uK}\FF(B)\leq\lambda_{uK}\FF(A),
\]
\[
\lambda_{lK}\FF(A)\leq\lambda_{lK}\FF(B).
\]
\end{lem}

Since $\bigvee_{x\notin\bigcap\A}f(x)=\bigvee_{x\in\bigcup_{A\in\A}(X\setminus A)}f(x)=\bigvee_{A\in\A}\bigvee_{x\notin A}f(x)$,
we have
\begin{equation}
\lambda_{uK}\FF(\bigcap_{A\in\A}A)=\bigvee_{A\in\A}\lambda_{uK}\FF(A).\label{eq:lambdauKofintersection}
\end{equation}

\begin{prop}
$\lambda_{lK}$ is a convergence approach structure on $\C_{c(\lambda)}$
and $\lambda_{uK}$ and $\lambda_{K}$ are preconvergence approach
structures on $\C_{c(\lambda)}$ that are centered on $\C_{r(\lambda)}$.
\end{prop}

\begin{proof}
If $\FF=\{A\}^{\uparrow_{\C_{c(\lambda)}}}$ then $\FF=\FF^{\#}$
and $(\rdc\FF)^{\uparrow}=\rdc(\FF^{\#})=\{A\}^{\uparrow_{X}}$. If
$A$ is $r(\lambda)$-closed, $\adh_{\lambda}A(x)=\infty$ for all
$x\notin A$. Hence, $1\oslash\adh_{\lambda}(\rdc\FF)(x)=0$ for all
$x\notin A$. Thus, $\lambda_{uK}(\{A\}^{\uparrow_{\C_{r(\lambda)}}})(A)=0$.
On the other hand, $\adh_{\lambda}A(x)=0$ for all $x\in A$ and $\lambda_{lK}(\{A\}^{\uparrow_{\C_{c(\lambda)}}})(A)=0$
for every $A\in\C_{c(\lambda)}$.

If $\FF\geq\GG$ then $\rdc\FF\geq\rdc\GG$ and $\GG^{\#}\supset\FF^{\#}$
so that $\rdc(\GG^{\#})\geq\rdc(\FF^{\#})$. Thus $\adh(\rdc\FF)\geq\adh(\rdc\GG)$
and $1\oslash\adh_{\lambda}(\rdc\GG)\geq1\oslash\adh_{\lambda}(\rdc\FF)$,
so that $\lambda_{uK}\GG\geq\lambda_{uK}\FF$. On the other hand,
$\adh(\rdc(\GG^{\#}))\geq\adh(\rdc(\FF^{\#}))$ and thus $\lambda_{lK}\GG\geq\lambda_{lK}\FF$.

The properties follow immediately for $\lambda_{K}$ on $\C_{r(\lambda)}$.
\end{proof}
\begin{prop}
\label{prop:Ktowers} If $(X,\lambda)$ is a convergence approach
space, $A\in\C_{c(\lambda)}$, $\epsilon\in[0,\infty]$, and $\FF$
is a filter on $\C_{c(\lambda)}$, then
\[
A\in\lm_{lK_{\epsilon}}\FF\iff A\subset\{\adh_{\lambda}\rdc(\FF^{\#})\leq\epsilon\}
\]
defines the tower $\{\lim_{lK_{\epsilon}}:\epsilon\in[0,\infty]\}$
associated with $\lambda_{lK}$ via (\ref{eq:CAPtotower}). Similarly,
\[
A\in\lm_{uK_{\epsilon}}\FF\iff\{\adh_{\lambda}(\rdc\FF)<1\oslash\epsilon\}\subset A,
\]
defines the tower $\{\lim_{uK_{\epsilon}}:\epsilon\in[0,\infty]\}$
associated with $\lambda_{uK}$ via (\ref{eq:CAPtotower}).
\end{prop}

\begin{proof}
By definition, $A\in\lm_{lK_{\epsilon}}\FF$ if $\bigvee_{x\in A}\adh_{\lambda}\rdc(\FF^{\#})\leq\epsilon$,
equivalently, if $A\subset\{\adh_{\lambda}\rdc(\FF^{\#})\leq\epsilon\}$.

Similarly, by definition, 
\[
A\in\lm_{uK_{\epsilon}}\FF\iff\bigvee_{x\notin A}1\oslash\adh_{\lambda}(\rdc\FF)(x)\leq\epsilon,
\]
 that is, $1\oslash\adh_{\lambda}(\rdc\FF)(x)\leq\epsilon$, equivalently,
$\adh_{\lambda}(\rdc\FF)(x)\geq1\oslash\epsilon$ for every $x\notin A$,
by (\ref{eq:divideineq}). In other words, $A\in\lm_{uK_{\epsilon}}\FF$
if and only if $\{\adh_{\lambda}(\rdc\FF)<1\oslash\epsilon\}\subset A$.
\end{proof}
\begin{thm}
If $(X,\lambda)$ is an approach space then $\lambda_{lK}$ is an
approach structure on $\C_{c(\lambda)}$.
\end{thm}

\begin{proof}
In view of Proposition \ref{prop:diagonaltower}, it is enough to
check that $\lim_{lK_{\epsilon}}$ is (pre)topological for every $\epsilon\in[0,\infty]$
and that the tower $\{\lim_{lK_{\epsilon}}:\epsilon\in[0,\infty]\}$
satisfies (\ref{eq:diagonaltower}). To this end, note that $(X,\lambda)$
can be given by a tower $(\tau_{\epsilon})_{\epsilon\in[0,\infty]}$
of pretopologies satisfying (\ref{eq:diagonaltower}). Moreover, in
view of Corollary \ref{cor:toweradh},
\[
A\in\lm_{lK_{\epsilon}}\FF\iff A\subset\{\adh_{\lambda}\rdc(\FF^{\#})\leq\epsilon\}=\adh_{\tau_{\epsilon}}\rdc(\FF^{\#})
\]
and $\lim_{lK_{\epsilon}}$ is the lower-Kuratowski pretopology associated
with the pretopology $\tau_{\epsilon}$ (see Proposition \ref{prop:pretoplK}). 

To see (\ref{eq:diagonaltower}), let $\SS:\C_{c(\lambda)}\to\mathbb{F}(\C_{c(\lambda)})$,
$A\in\C_{c(\lambda)}$ and $\FF\in\mathbb{F}(\C_{c(\lambda)})$ and
suppose that $C\in\lm_{lK_{\epsilon}}\SS(C)$ for every $C\in\C_{c(\lambda)}$
and $A\in\lm_{lK_{\gamma}}\FF$. In view of Proposition \ref{prop:pretoplK},
this means that $\{V^{-}:V\in\V_{\tau_{\epsilon}}(t),t\in C\}\subset\SS(C)$
and $\{W^{-}:W\in\V_{\tau_{\gamma}}(x),x\in A\}\subset\FF$. 

We want to show that $\{U^{-}:U\in\V_{\tau_{\epsilon+\gamma}}(x),x\in A\}\subset\SS(\FF)$.
Since $(\tau_{\epsilon})_{\epsilon\in[0,\infty]}$ is the tower for
an approach space, it satisfies (\ref{eq:diagonaltower}) and thus
\[
\V_{\tau_{\epsilon}}(\V_{\tau_{\gamma}}(x))\geq\V_{\tau_{\epsilon+\gamma}}(x)
\]
for every $x\in A$, so that for every $U\in\V_{\tau_{\epsilon+\gamma}}(x)$
there is $W\in\V_{\tau_{\gamma}}(x)$ with $U\in\bigcap_{t\in W}\V_{\tau_{\epsilon}}(t)$,
hence $U^{-}\in\SS(C)$ whenever $C\in W^{-}$, that is, $U^{-}\in\bigcap_{C\in W^{-}}\SS(C)$.
As $W^{-}\in\FF$, we conclude that $U^{-}\in\bigcup_{\F\in\FF}\bigcap_{C\in\F}\SS(C)=\SS(\FF)$.
\end{proof}
\begin{thm}
\label{thm:convreflcorefl} The $\conv$-coreflection of
\begin{enumerate}
\item \label{enu:lK} $\lambda_{lK}$ is the lower-Kuratowski convergence
$\lm_{lK}$ associated with $c(\lambda)$ provided that $(X,\lambda)$
is a pre-approach space;
\item \label{enu:uK} $\lambda_{uK}$ is the upper-Kuratowski convergence
$\lim_{uK}$ associated with $r(\lambda)$.
\end{enumerate}
Moreover, the $\conv$-reflection of 
\begin{enumerate}
\item $\lambda_{lK}$ is the lower-Kuratowski convergence $\lm_{lK}$ associated
with $r(\lambda)$;
\item $\lambda_{uK}$ is the upper Kuratowski convergence $\lim_{uK}$ associated
with $c(\lambda)$ provided that $(X,\lambda)$ is a pre-approach
space.
\end{enumerate}
Hence, if $(X,\lambda)$ is a topological approach space, $\lambda=i(c(\lambda))=i(r(\lambda))$
and $c(\lambda_{K})=r(\lambda_{K})=\lim_{K}$ is the Kuratowski convergence.
\end{thm}

\begin{proof}
(\ref{enu:lK}). If $(X,\lambda)$ is a pre-approach space, then in
view of Proposition \ref{prop:adhpreapcorefl}, $\lambda_{lK}\FF(A)=0$
if and only if 
\[
A\subset\{\adh_{\lambda}\rdc(\FF^{\#})=0\}=\adh_{c(\lambda)}\rdc(\FF)^{\#},
\]
so that $\lambda_{lK}\FF(A)=0$ if and only if $A\in\lm_{lK}\FF$
for $c(\lambda)$.

(\ref{enu:uK}). In view of Proposition \ref{prop:Ktowers}, $\lambda_{uK}\FF(A)=0$
if and only $\{\adh_{\lambda}(\rdc\FF)<\infty\}\subset A$, if and
only if $\adh_{r(\lambda)}(\rdc\FF)\subset A$, equivalently, $A\in\lm_{uK}\FF$
for the convergence space $(X,r(\lambda))$. 

On the other hand, $\lambda_{lK}\FF(A)<\infty$ if and only if 
\[
A\subset\{\adh_{\lambda}\rdc(\FF^{\#})<\infty\}=\adh_{r(\lambda)}\rdc(\FF)^{\#}
\]
 and thus $\lambda_{lK}\FF(A)<\infty$ if and only if $A\in\lm_{lK}\FF$
for $(X,r(\lambda))$.

Finally, suppose that $(X,\lambda)$ is pre-approach. By Proposition
\ref{prop:Ktowers}, $\lambda_{uK}\FF(A)<\infty$ if and only if $\{\adh_{\lambda}(\rdc\FF)=0\}\subset A$.
By Proposition \ref{prop:adhpreapcorefl}, $\{\adh_{\lambda}(\rdc\FF)=0\}=\adh_{c(\lambda)}\rdc\FF$.
Hence $\lambda_{uK}\FF(A)<\infty$ if and only if $A\in\lm_{uK}\FF$
for $(X,c(\lambda))$.
\end{proof}
\begin{lem}
If $(X,\lambda)$ is an approach space, $\A\subset\C_{c(\lambda)}$,
and the family $(\FF_{A})_{A\in\A}$ of filters on $\C_{c(\lambda)}$
is directed, then
\[
\lambda_{uK}(\bigvee_{A\in\A}\FF_{A})(\bigcap\A)\leq\bigvee_{A\in\A}\lambda_{uK}(\FF_{A})(A).
\]
\end{lem}

\begin{thm}
\label{thm:weakdiagwhenuKpreapp} Let $(X,\lambda)$ be an approach
space. If $\lambda_{uK}$ is a pre-approach on $\C_{r(\lambda)}$,
then for every $\FF\in\mathbb{F}(\C_{r(\lambda)})$, $\adh_{r(\lambda)}\rdc\FF$
is a point of $\FF$-diagonality of $\lambda_{uK}$.
\end{thm}

\begin{proof}
Suppose $\lambda_{uK}$ is pre-approach and let $A\in\C_{r(\lambda)}$
and $\FF\in\mathbb{F}\C_{r(\lambda)}$. We want to show that if $\GG:\C_{r(\lambda)}\to\mathbb{F}\C_{r(\lambda)}$
then
\[
\lambda_{uK}(\GG(\FF))(A)\leq\lambda\FF(A)+\bigvee_{C\in\C_{r(\lambda)}}\lambda_{uK}(\GG(C))(C).
\]

For every $\A\in\FF$ and every $C\in\A$, $\lambda_{uK}(\GG(C))(\cl_{r(\lambda)}(\bigcup_{C\in\A}C))\leq\lambda_{uK}(\GG(C))(C)$
so that, because $\lambda_{uK}$ is pre-approach,
\[
\lambda_{uK}(\bigcap_{C\in\A}\GG(C))(\cl_{r(\lambda)}(\bigcup_{C\in\A}C))=\bigvee_{C\in\A}\lambda_{uK}(\GG(C))(\cl_{r(\lambda)}(\bigcup_{C\in\A}C))\leq\bigvee_{C\in\A}\lambda_{uK}(\GG(C))(C).
\]

In view of (\ref{eq:lambdauKofintersection}),
\begin{eqnarray*}
\lambda_{uK}(\bigvee_{\A\in\FF}\bigcap_{C\in\A}\GG(C))(\bigcap_{\A\in\FF}\cl_{r(\lambda)}(\bigcup_{C\in\A}C)) & \leq & \lambda_{uK}(\bigcap_{C\in\A}\GG(C))(\bigcap_{\A\in\FF}\cl_{r(\lambda)}(\bigcup_{C\in\A}C))\\
 & \leq & =\bigvee_{\A\in\FF}\bigvee_{C\in\A}\lambda_{uK}(\GG(C))(C),
\end{eqnarray*}
that is,
\begin{eqnarray*}
\lambda_{uK}(\GG(\FF))(\adh_{r(\lambda)}\rdc\FF) & \leq & \bigvee_{C\in\C_{r(\lambda)}}\lambda_{uK}(\GG(C))(C)\\
 & \leq & \lambda_{uK}\FF(\adh_{r(\lambda)}\rdc\FF)+\bigvee_{C\in\C_{r(\lambda)}}\lambda_{uK}(\GG(C))(C),
\end{eqnarray*}
because $\lambda\FF(\adh_{r(\lambda)}\rdc\FF)=0$.
\end{proof}
\begin{cor}
If $X$ is a topological space and the upper Kuratowski convergence
on $\C_{X}$ is pretopological, then it is topological.
\end{cor}

\begin{proof}
If $(X,\xi)$ is topological, $(X,\lambda_{\xi})$ is a (topological)
approach space and the corresponding $\lambda_{uK}$ is given by
\[
\lambda_{uK}\FF(A)=\bigvee_{x\notin A}1\oslash\adh_{\lambda}(\rdc\FF)(x)=\begin{cases}
\infty & \adh_{\xi}\rdc\FF\cap A^{c}\neq\emptyset\\
0 & \adh_{\xi}\rdc\FF\subset A
\end{cases}
\]
and thus coincides with $\lambda_{\lim_{uK}}$, which is pre-approach
because $\lim_{uK}$ is pretopological. By Theorem \ref{thm:weakdiagwhenuKpreapp},
for every $\FF\in\mathbb{F}(\C_{r(\lambda)})$, $\adh_{r(\lambda)}\rdc\FF=\adh_{\xi}\rdc\FF$
is a point of $\FF$-diagonality, that is, $\adh_{\xi}\rdc\FF\in\lm_{uK}\GG(\FF)$
for every $\GG:\C_{\xi}\to\mathbb{F}\C_{\xi}$ with $A\in\lim_{uK}\GG(A)$.
As a result, $C\in\lm_{uK}\GG(\FF)$ whenever $C\in\lm_{uK}\FF$
and thus $\lim_{uK}$ is topological. 
\end{proof}
This result, though often formulated in different terms, can be traced
back to \cite{DayKelly} and is well-known and re-appears under different
forms in the literature, e.g., \cite{schwarz.powers,compedium}, \cite[Theorem 5.6]{DM.products}
though usually through indirect arguments; a direct argument is provided
for instance in \cite[Proposition 5.3]{JM.corecompact}.

\section{Cocompact and Fell}

Recall from Section \ref{subsec:Hyperspace-convergences-and} that
the cocompact topology or upper-Fell topology on the set $\C_{X}$
of closed subsets of a topological space is generated by the subbase
given by sets of the form
\[
K^{+}=\{A\in\C_{X}:A\cap K=\emptyset\},
\]
where $K$ runs over compact subsets of $X$, so that $A\in\lim_{uF}\FF$
if 
\[
A\in K^{+}\then K^{+}\in\FF,
\]
that is, if 
\[
A\cap K=\emptyset\then K\notin(\rdc\FF)^{\#},
\]

for every compact subset $K$ of $X$, equivalently,
\[
K\subset A^{c},K\#\rdc\F\then K\text{ not compact.}
\]
In other words, for such $K$'s, the closer to compact, the worse
the convergence. On this basis, we define the \emph{upper-Fell convergence
approach structure} on $\C_{c(\lambda)}$ or on its subset $\C_{r(\lambda)}$
by
\begin{equation}
\lambda_{uF}\FF(A)=\bigvee_{H\in(\rdc\FF)^{\#},H\subset A^{c}}1\oslash m(H),\label{eq:lambdauF}
\end{equation}
where 
\[
m(H)=\bigvee_{\U\in\beta H}\bigwedge_{x\in H}\lambda\U(x)=\bigvee_{H\in\H^{\#}}\bigwedge_{x\in H}\adh_{\lambda}\H(x)
\]
 is the measure of compactness of $H$ in $(X,\lambda)$. 
\begin{thm}
\label{thm:lambdauFpreAP} If $(X,\lambda)$ is a convergence approach
space then $\lambda_{uF}$ is a non-Archimedean approach structure
on $\C_{X}$ satisfying $\lambda_{uK}\geq\lambda_{uF}$.
\end{thm}

\begin{proof}
First note that $\lambda_{uF}$ is centered for $(\rdc\{A\}^{\uparrow_{\C_{X}}})^{\uparrow_{X}}=A^{\uparrow_{X}}$
so that $H\subset A^{c}$ and $H\in(\rdc\{A\}^{\uparrow_{\C_{X}}})^{\#}$
are incompatible, so that $\lambda_{uF}(\{A\}^{\uparrow_{\C_{X}}})(A)=0$.
Since 
\global\long\def\DD{\mathfrak{D}}%
\[
\rdc(\bigcap_{\DD\in\mathbb{D}}\DD)=\bigcap_{\DD\in\mathbb{D}}\rdc\DD,
\]
we have $\left(\rdc(\bigcap_{\DD\in\mathbb{D}}\DD)\right)^{\#}=\bigcup_{\DD\in\mathbb{D}}(\rdc\DD)^{\#}$
and thus $\lambda_{uF}(\bigcap_{\DD\in\mathbb{D}}\DD)(A)=\bigvee\lambda_{uF}(\DD)(A)$,
that is, $\lambda_{uF}$ is a pre-approach structure. 

To see that $\lambda_{uF}$ is non-Archimedean approach structure,
we show (\ref{eq:nonarchimedean}) for $\lambda_{uF}$. Let $\GG(\cdot):\C_{X}\to\mathbb{F}(\C_{X})$,
$A_{0}\in\C_{X}$ and $\FF\in\mathbb{F}(\C_{X})$. We need to show
that $1\oslash m(H)\leq\lambda_{uF}\FF(A_{0})\vee\bigvee_{C\in\C_{X}}\lambda_{uF}(\GG(C))(C)$
whenever $H\subset A_{0}^{c}$ and $H\in(\rdc(\GG(\FF))^{\#}$. By
\cite[Prop. VI.1.4]{DM.book},
\[
(\rdc(\GG(\FF))^{\#}=((\rdc\GG)(\FF))^{\#}=(\rdc\GG)^{\#}(\FF^{\#}),
\]
so that there is $\H\in\FF^{\#}$ such that $H\in\bigcap_{C\in\H}(\rdc\GG(C))^{\#}$.
If there is $C\in\H$ disjoint from $H$, then $1\oslash m(H)\leq\lambda_{uF}(\GG(C))(C)\leq\bigvee_{C\in\C_{X}}\lambda_{uF}(\GG(C))(C)$.
Else, $H\cap C\neq\emptyset$ whenever $C\in\H$ and $\H\in\FF^{\#}$
so that $H\in(\rdc\FF)^{\#}$. Thus $1\oslash m(H)\leq\lambda_{uF}(\FF)(A_{0})$
and the proof that $\lambda_{uF}$ is non-Archimedean is complete.

Finally, $\lambda_{uK}\geq\lambda_{uF}$, because if $H\in(\rdc\F)^{\#},H\subset A^{c}$,
then $\bigwedge_{x\notin A}\adh_{\lambda}(\rdc\FF)(x)\leq\bigwedge_{x\in H}\adh_{\lambda}(\rdc\FF)(x)\leq m(H)$
so $1\oslash m(H)\leq1\oslash\bigwedge_{x\notin A}\adh_{\lambda}(\rdc\FF)(x)$.
\end{proof}
On the other hand the upper and lower Vietoris approach structures
were introduced by Lowen and his collaborators. We follow here the
presentation from \cite{indextheory}. We start from an approach space
$X$ given by its lower regular function frame $\L=\cont(X,[0,\infty])$.
Note that $\mu\in\cont(X,[0,\infty])$ lifts to maps
\begin{eqnarray*}
\mu^{\wedge}:\C_{cX} & \to & [0,\infty]\\
A & \mapsto & \bigwedge\mu(A)
\end{eqnarray*}
and 
\begin{eqnarray*}
\mu^{\vee}:\C_{cX} & \to & [0,\infty]\\
A & \mapsto & \bigvee\mu(A).
\end{eqnarray*}

Such maps generate lower regular function frames on $\C_{cX}$ given
by
\[
\L^{\wedge}=\{\sup_{\mu\in J}\mu^{\wedge}:J\subset\L\}
\]
and 
\[
\L^{\vee}=\{\sup_{j\in J}\inf_{k\in K_{j}}\mu_{j,k}^{\vee}:K_{j}\text{ finite},\mu_{j,k}\in\L\}.
\]

They define \emph{the Vietoris $\wedge$-approach structure }and\emph{
Vietoris $\vee$-approach structure }on $\C_{cX}$, which coincide
with the upper and lower Vietoris topologies on $\C_{cX}$ respectively,
when $X$ is topological \cite[Prop. 10.3.2]{indextheory}.

The corresponding limit operators are 
\begin{eqnarray*}
\lambda_{uV}\FF(A) & = & \bigvee_{\H\in\FF^{\#}}\adh_{\L^{\wedge}}\H(A)=\bigvee_{\H\in\FF^{\#}}\bigvee\{g(A):g\in\L^{\wedge},g_{|\H}=0\}\\
 & = & \bigvee_{\mu\in\L}\left(\mu^{\wedge}(A)\ominus\bigvee_{\F\in\FF}\bigwedge_{C\in\F}\mu^{\wedge}(C)\right)
\end{eqnarray*}
and 
\begin{eqnarray*}
\lambda_{lV}\FF(A) & = & \bigvee_{\H\in\FF^{\#}}\adh_{\L^{\vee}}\H(A)=\bigvee_{\H\in\FF^{\#}}\bigvee\{g(A):g\in\L^{\vee},g_{|\H}=0\}.\\
 & = & \bigvee_{\mu\in\L}\left(\mu^{\vee}(A)\ominus\bigvee_{\F\in\FF}\bigwedge_{C\in\F}\mu^{\vee}(C)\right)=\bigvee_{x\in A}\bigvee_{\mu\in\L}\left(\mu(x)\ominus\bigvee_{\F\in\FF}\bigwedge_{C\in\F}\mu^{\vee}(C)\right)
\end{eqnarray*}

Note that in the traditional topological setting, the lower Vietoris
and lower Kuratowski topologies coincide, e.g., \cite[Proposition VIII.7.5]{DM.book}.
However, in general $\lambda_{lK}$ is finer than $\lambda_{lV}$,
and this inequality may be strict, though we have coincide over approach
spaces. 
\begin{thm}
\label{thm:lKfinerthanlV} If $(X,\lambda)$ is a convergence approach
space and $\lambda_{lK}$ and $\lambda_{lV}$ denote the corresponding
lower-Kuratowski and lower-Vietoris convergence approach structures
on $\C_{c(\lambda)}$, then 
\[
\lambda_{lK}\geq\lambda_{lV}.
\]
Moreover, this is an equality if $(X,\lambda)$ is an approach space.
\end{thm}

\begin{proof}
For the first part, it is enough to show that $\mu^{\vee}\in\cont(\lambda_{lK},V)$
whenever $\mu\in\cont(\lambda,V)$. Let $\FF\in\mathbb{F}\C_{c(\lambda)}$
and $A\in\C_{c(\lambda)}$. We need to show, 
\[
\mu^{\vee}(A)\leq\lambda_{lK}\FF(A)+\bigvee_{F\in\FF}\bigwedge_{C\in F}\mu^{\vee}(C).
\]
Let $\alpha=\lambda_{lK}\FF(A)=\bigvee_{x\in A}\adh_{\lambda}(\rdc\FF^{\#})(x)$
and $\gamma=\bigvee_{F\in\FF}\bigwedge_{C\in F}\mu^{\vee}(C)$.

For every $x\in A$ and $\epsilon>0$, $\adh_{\lambda}(\rdc\FF^{\#})(x)\leq\alpha$
so that there is a filter $\G_{x}\#\rdc(\FF^{\#})$ with $\lambda(\G_{x})(x)\leq\alpha+\epsilon$.
In view of (\ref{eq:BingrillofrdcFgrill}), $\{G^{-}:G\in\G_{x}\}\subset\FF$.
Hence, for every $G\in\G_{x}$, 
\[
\gamma=\bigvee_{F\in\FF}\bigwedge_{C\in F}\mu^{\vee}(C)\geq\bigwedge_{C\in G^{-}}\mu^{\vee}(C)
\]
 and for every $C\in G^{-}$, there is $z\in C\cap G$, so that 
\[
\mu^{\vee}(C)=\sup_{x\in C}\mu(x)\geq\mu(z)\geq\inf_{w\in G}\mu(w).
\]
Hence, $\gamma\geq\bigwedge_{C\in G^{-}}\mu^{\vee}(C)\geq\inf_{w\in G}\mu(w)$
for every $G\in\G_{x}$ and thus
\[
\gamma\geq\bigvee_{G\in\G_{x}}\bigwedge_{C\in G^{-}}C\geq\sup_{G\in\G_{x}}\inf_{w\in G}\mu(w).
\]
Moreover, as $\mu\in\cont(\lambda,[0,\infty])$, 
\[
\mu(x)\leq\lambda(\G_{x})(x)+\sup_{G\in\G_{x}}\inf_{w\in G}\mu(w)\leq\alpha+\epsilon+\gamma.
\]
As $\epsilon>0$ is arbitrary, we conclude that $\mu(x)\leq\alpha+\gamma$
for all $x\in A$, and thus $\mu^{\vee}(A)\leq\alpha+\gamma$, as
required.

Now, suppose that $(X,\lambda)$ is an approach space with regular
function frame $\L$. In this case 
\[
\lambda_{lK}\FF(A)=\bigvee_{x\in A}\adh_{\lambda}(\rdc\FF^{\#})(x)=\bigvee_{x\in A}\bigvee_{\G\in\FF^{\#}}\adh_{\lambda}\rdc\G(x).
\]
Since 
\[
\lambda_{lV}\FF(A)=\bigvee_{x\in A}\bigvee_{\mu\in\L}\mu(x)\ominus\bigvee_{\F\in\FF}\bigwedge_{C\in\F}\mu^{\vee}(C)
\]
it is enough to show that 
\[
\bigvee_{\mathcal{G}\in\mathfrak{F}^{\#}}\adh_{\lambda}\rdc\G(x)=\bigvee_{\mu\in\mathcal{L}}\left(\mu(x)\ominus\bigvee_{\F\in\FF}\bigwedge_{C\in\F}\mu^{\vee}(C)\right)
\]
for every $x$. To this end, Let $\alpha=\bigvee_{\mathcal{G}\in\mathfrak{F}^{\#}}\adh_{\lambda}\rdc\G(x)$.
It is enough to show that for $\gamma<\alpha$ there is $\mu\in\L$
with $\mu(x)\ominus\bigvee_{\F\in\FF}\bigwedge_{C\in\F}\mu^{\vee}(C)\geq\gamma$.
Given such a $\gamma$ there is $\G_{0}\in\FF^{\#}$ with 
\[
\adh_{\lambda}\rdc\G_{0}(x)=\bigvee_{\mu\in\L,\mu_{|\rdc\G_{0}}=0}\mu(x)>\gamma,
\]
so that there is $\mu\in\L$ with $\mu_{|\rdc\G_{0}}=0$ and $\mu(x)>\gamma$.
Since $\G_{0}\in\FF^{\#}$, for every $\F\in\FF$ there is $K_{\F}\in\F\cap\G_{0}$
so that $\mu_{|K_{\F}}=0$ because $\mu_{|\rdc\G_{0}}=0$ and $K_{\F}\subset\rdc\G_{0}$.
Hence $\mu^{\vee}(K_{\F})=0$ and thus $\bigwedge_{C\in\F}\mu^{\vee}(C)=0$
and $\bigvee_{\F\in\FF}\bigwedge_{C\in\F}\mu^{\vee}(C)=0$. Therefore,
for this specific $\mu$, 
\[
\mu(x)\ominus\bigvee_{\F\in\FF}\bigwedge_{C\in\F}\mu^{\vee}(C)=\mu(x)\ominus0=\mu(x)>\gamma
\]
as desired. 
\end{proof}
Moreover, the approach structure $\lambda_{lK}$ may be strictly finer
than the approach structure $\lambda_{lV}$. This phenomenon occurs
already in the classical setting of $\conv$, as soon as the base
space $X$ is not topological. 
\begin{example}[$\lim_{lK}\neq\lim_{lV}$]
 Take the ``simplest'' non-topological pretopology on $\{a,b,c\}$
defined by $\V(a)=\{a\}^{\uparrow}$, $\V(b)=\{a,b\}^{\uparrow}$
and $\V(c)=\{b,c\}^{\uparrow}$ so that $\lim\{a\}^{\uparrow}=\{a,b\}$,
$\lim\{b\}^{\uparrow}=\{b,c\}$ and $\lim\{c\}^{\uparrow}=\{c\}$.
Then $\C_{X}=\{\emptyset,\{c\},\{b,c\},X\}$. If $\FF=\{\{a\}\}^{\uparrow}=\FF^{\#}$
then $A\in\lm_{lV}\FF$ if for every $O\in\O_{X}$,
\[
O\cap A\neq\emptyset\then O^{-}\in\FF,
\]
equivalently (using the contrapositive and $C=X\setminus O$), 
\[
a\in C\then A\subset C
\]
for every $C\in\C_{X}$. But the only closed set containing $a$ is
$X$ so $\C_{X}\subset\lm_{lV}\FF$. On the other hand, $A\in\lm_{lK}\FF$
if and only if $A\subset\adh(\rdc\FF)=\adh\{a\}=\{a,b\}$ but the
only closed set included in $\{a,b\}$ is the empty set, so that $\lim_{lK}\FF=\{\emptyset\}$.
\end{example}

Similarly, define for every $B\in\C_{c(\lambda)}$, 
\begin{eqnarray*}
\mu_{B}^{\wedge}:\C_{c(\lambda)} & \to & [0,\infty]\\
A & \mapsto & \bigwedge\mu(A\cap B)
\end{eqnarray*}
and 
\begin{eqnarray*}
\mu_{B}^{\vee}:\C_{c(\lambda)} & \to & [0,\infty]\\
A & \mapsto & \bigvee\mu(A\cap B).
\end{eqnarray*}

The corresponding lower regular function frame $\L_{uF}$ generated
by $\{\mu_{B}^{\wedge}:\mu\in\L,m(B)=0\}$ defines the upper-Fell
approach structure in the sense of \cite{atecs2023fell}. 
\begin{thm}
If $(X,\lambda)$ is a convergence approach space then $\lambda_{uF}$
defined on $\C_{c(\lambda)}$ by (\ref{eq:lambdauF}) is finer than
the approach structure defined by $\L_{uF}$.
\end{thm}

\begin{proof}
It is enough to show that $\mu_{B}^{\wedge}\in\cont(\lambda_{uF},\lambda_{V})$
whenever $\mu\in\L_{\lambda}$ and $B=\cl_{c(\lambda)}B$ is a subset
of $X$ with $m(B)=0$. Let $\nu:=\mu_{B}^{\wedge}$. Assume to the
contrary that $\mu_{B}^{\wedge}\notin\cont(\lambda_{uF},\lambda_{V})$,
so that there is a filter $\FF$ on $\C_{c(\lambda)}$ and $A\in\C_{c(\lambda)}$
for which $\lambda_{uF}(\FF)(A)<\lambda_{V}(\nu[\FF])(\nu(A))$. In
particular, $\alpha:=\lambda_{uF}(\FF)(A)$ and $\bigwedge_{H\in\nu[\FF]^{\#}}\bigvee H$
are both finite and there is $\gamma$ with 
\begin{equation}
\bigwedge_{H\in\nu[\FF]^{\#}}\bigvee H<\gamma<\nu(A)\ominus\alpha\leq\nu(A).\label{eq:auxlambdauFvsLuF}
\end{equation}
Hence there is $H\in\nu[\FF]^{\#}$ with $\bigvee H<\gamma$. Then
$\H=\nu^{-1}(H)\in\FF^{\#}$ and $\bigvee_{C\in\H}\nu(C)<\gamma$.
Hence, for every $\F\in\FF$ there is $C_{\F}\in\F$ with $\nu(C_{\F})<\gamma$,
so that there is $x_{\F}\in C_{\F}\cap B$ with $\mu(x_{\F})<\gamma$.
Let $H:=\{x_{\F}:\F\in\FF\}$. Note that $H\subset B$ and $H\in(\rdc\FF)^{\#}$.
As $B\in\C_{c(\lambda)}$, $H\subset B$ and $m(B)=0$, the set $K=\cl_{c(\lambda)}H$
is a subset of $B$ satisfying $m(K)=0$ and $K\in(\rdc\FF)^{\#}$.
If we had $K\subset A^{c}$, then in view of (\ref{eq:lambdauF}),
we would have $\lambda_{uF}(\FF)(A)=\infty$, which we already noted
is not the case. Hence, there is $z\in K\cap A\subset B\cap A$, so
that $\nu(A)=\inf_{x\in A\cap B}\mu(x)\leq\mu(z)$. Moreover, $z\in\cl_{c(\lambda)}H\cap A=H^{(0)}\cap A$
while $\mu(x)<\gamma$ for every $x\in H$ and $\mu\in\cont(\lambda,\lambda_{V})$,
so that $\mu(z)\leq\gamma$. Hence $\nu(A)\leq\gamma$, in contradiction
with (\ref{eq:auxlambdauFvsLuF}).
\end{proof}
Moreover, this inequality is strict in general:
\begin{example}[The (non-Archimedean) approach structure $\lambda_{uF}$ may be strictly
finer than the approach structure given by $\L_{uF}$]

Consider on $X=\mathbb{N}$ the approach structure $\lambda$ given
by 
\[
\lambda(\F)(x)=\begin{cases}
0 & \text{if }\text{\ensuremath{\F}}=\{x\}^{\uparrow}\\
1 & \text{if }\F\text{ is free}
\end{cases},
\]
in which it is plain that $m(B)=0$ if and only if $B$ is finite,
and $m(B)=1$ whenever $B$ is infinite. Let $\FF$ be the filter
generated on $\C_{X}$ by the sequence of singletons $\{\{n\}\}_{n=1}^{\infty}$.
Note that $(\rdc\FF)^{\uparrow}=\{\{n\geq k\}:k\in\mathbb{N}\}^{\uparrow}$
is the cofinite filter on $X$, so that $H\in(\rdc\FF)^{\#}$ if and
only if $H$ is infinite, if and only if $m(H)=1$. Hence,
\[
\lambda_{uF}\FF(\emptyset)=\bigvee_{H\in(\rdc\FF)^{\#}}1\oslash m(H)=1.
\]

On the other hand, 
\[
\lambda_{\L_{uF}}\FF(\emptyset)=\bigvee_{\nu\in\L_{uF}}\lambda_{V}(\nu[\FF])(\nu(\emptyset))=0
\]
because if $m(B)=0$, then $B$ is finite and thus $\nu=\mu_{B}^{\wedge}$
satisfies $\nu[\FF]=\{\infty\}^{\uparrow}$ as $\{\{n\}:n\geq k\}\cap B=\emptyset$
for $k$ large enough, so that $\nu(\{\{n\}:n\geq k\})=\inf_{\emptyset}\mu=\infty$. 
\end{example}

Now the \emph{Fell convergence approach structure} can be defined
as 
\[
\lambda_{\bar{F}}\FF(A)=\lambda_{uF}\FF(A)\vee\lambda_{lK}\FF(A)
\]
or alternatively
\[
\lambda_{F}\FF(A)=\lambda_{uF}\FF(A)\vee\lambda_{lV}\FF(A)
\]
to obtain a structure that coincide with the Fell convergence when
$(X,\lambda)$ is topological. In view of Theorem \ref{thm:lambdauFpreAP}
and Theorem \ref{thm:lKfinerthanlV},
\[
\lambda_{K}\geq\lambda_{\bar{F}}\geq\lambda_{F}.
\]

The study of conditions under which some of the hyperspace convergence
approach structures introduced in this paper may coincide will be
the topic of future research.
\begin{rem}
Of course, all (pre)convergence approach hyperspace structures considered
here can be presented in terms of nets. Namely, if $\lambda_{H}$
is one of the associated (pre)convergence approach space structure
on $\C_{c(\lambda)}$ considered in these notes (e.g., $\lambda_{uK},\lambda_{lK},\lambda_{lV},\lambda_{uV},\lambda_{uF},\lambda_{F}$),
and $\A=\left\langle A_{\alpha}\right\rangle _{\alpha\in\Lambda}$
is a net on $\C_{c(\lambda)}$, we define the convergence approach
of $\A$ in $\lambda_{H}$ in terms of the convergence approach of
the associated filter $\FF_{\A}=\{\{A_{\alpha}:\alpha\geq_{\Lambda}\beta\}:\beta\in\Lambda\}^{\uparrow_{\C_{c(\lambda)}}}$
via
\[
\lambda_{H}(\A)(C):=\lambda_{H}(\FF_{\A})(C),
\]
as particular case of Remark \ref{rem:netsinCAP}.
\end{rem}





\begin{thebibliography}{10}
	
	\bibitem{atecs2023fell}
	M.~Ate{\c{s}} and S.~Sa{\u{g}}{\i}ro{\u{g}}lu Peker.
	\newblock The {F}ell approach structure.
	\newblock {\em Communications Faculty of Sciences University of Ankara Series
		A1 Mathematics and Statistics}, 72(3):633--649, 2023.
	
	\bibitem{Beer}
	G.~Beer.
	\newblock {\em Topologies on Closed and Closed Convex Sets}.
	\newblock Kluwer Academic, 1993.
	
	\bibitem{towers}
	P.~Brock and D.~C. Kent.
	\newblock Approach spaces, limit tower spaces, and probabilistic convergence
	spaces.
	\newblock {\em Appl. Cat. Struct.}, 5:99--110, 1997.
	
	\bibitem{DayKelly}
	B.~J. Day and G.~M. Kelly.
	\newblock On topological quotient maps preserved by pullbacks or products.
	\newblock {\em Proc. Camb. Phil. Soc}, {\bf 67}:553--558, 1970.
	
	\bibitem{DM.products}
	S.~Dolecki and F.~Mynard.
	\newblock Convergence-theoretic mechanisms behind product theorems.
	\newblock {\em Topology and its Applications}, {\bf 104}:67--99, 2000.
	
	\bibitem{DM.book}
	S.~Dolecki and F.~Mynard.
	\newblock {\em Convergence {F}oundations of {T}opology}.
	\newblock World Scientific, 2016.
	
	\bibitem{compedium}
	G.~Gierz, K.~H. Hofmann, K.~Keimel, J.~Lawson, M.~Mislove, and D.~Scott.
	\newblock {\em A Compedium of Continuous Lattices}.
	\newblock Springer-Verlag, Berlin, 1980.
	
	\bibitem{JM.corecompact}
	F.~Jordan and F.~Mynard.
	\newblock Core compactness and diagonality in spaces of open sets.
	\newblock {\em Applied General Topology}, 12(2):143--162, 2011.
	
	\bibitem{Lowen88}
	E.~Lowen and R.~Lowen.
	\newblock A quasitopos containing {CONV} and {MET} as full subcategories.
	\newblock {\em Int. J. Math. Sci.}, 19:417--438, 1988.
	
	\bibitem{lowe88}
	E.~Lowen and R.~Lowen.
	\newblock Topological quasitopos hulls of categories containing topological and
	metric objects.
	\newblock {\em Cahiers de Topologies et G{\'e}ometrie diff{\'e}rentielle
		Cat{\'e}gorique}, 30:213--228, 1989.
	
	\bibitem{Lowen97}
	E.~Lowen, R.~Lowen, and C.~Verbeeck.
	\newblock Exponential objects in {PRAP}.
	\newblock {\em Cahiers de Topologies et G{\'e}ometrie diff{\'e}rentielle
		Cat{\'e}gorique}, 38:259--276, 1997.
	
	\bibitem{AP.book}
	R.~Lowen.
	\newblock {\em Approach Spaces: The Missing Link in Topology-Uniformity-Metric
		Triad}.
	\newblock Oxford University Press, 1997.
	
	\bibitem{indextheory}
	R.~Lowen.
	\newblock {\em {I}ndex {A}nalysis: {A}pproach {T}heory at {W}ork}.
	\newblock Springer Monographs in Mathematics. Springer, 2015.
	
	\bibitem{myn.APinCAP}
	F.~Mynard.
	\newblock On (pre)-approach spaces within convergence approach spaces.
	\newblock {\em submitted}.
	
	\bibitem{mynardmeasureCAP}
	F.~Mynard.
	\newblock Measure of compactness for filters in the approach setting.
	\newblock {\em Quaestiones Math.}, 31:189--201, 2008.
	
	\bibitem{schwarz.powers}
	F.~Schwarz.
	\newblock Powers and exponential objects in initially structured categories and
	application to categories of limits spaces.
	\newblock {\em Quaest. Math.}, {\bf 6}:227--254, 1983.
	
\end{thebibliography}
\end{document}